\documentclass[11pt, a4paper]{amsart}
\usepackage[T1]{fontenc}

\usepackage[margin=3cm, top=4cm, bottom=3cm]{geometry}
  
\usepackage[dvipsnames]{xcolor}
\usepackage{xcolor-solarized}

\usepackage{calc} 

\usepackage{amsmath,amssymb}
\usepackage{mathtools}

\usepackage{graphicx} 
\usepackage{subfig}

\usepackage[colorlinks=true, allcolors=blue]{hyperref}
\usepackage{cleveref}

\usepackage{enumitem}
\usepackage{comment}

\usepackage[arrow, matrix, curve]{xy}
\usepackage{tikz}
\usetikzlibrary{
  cd,
  calc,
  positioning,
  fit,
  arrows,
  decorations.pathreplacing,
  decorations.markings,
  shapes.geometric,
  backgrounds,
  bending
}
\usepackage{tikzsymbols}

\usepackage{newtxtext,newtxmath}

\newcommand{\R}{\mathbb{R}}

\newcommand{\cB}{\mathcal{B}}

\newcommand{\cE}{\mathcal{E}}
\newcommand{\cF}{\mathcal{F}}

\newcommand{\cJ}{\mathcal{J}}

\newcommand{\cM}{\mathcal{M}}
\newcommand{\cN}{\mathcal{N}}
\newcommand{\cO}{\mathcal{O}}
\newcommand{\cP}{\mathcal{P}}

\newcommand{\cS}{\mathcal{S}}
\newcommand{\cT}{\mathcal{T}}

\newcommand{\cZ}{\mathcal{Z}}

\let\ker\relax
\DeclareMathOperator{\ker}{Ker}

\DeclareMathOperator{\im}{Im}

\DeclareMathOperator{\supp}{Supp}
\DeclareMathOperator{\End}{End}
\DeclareMathOperator{\wind}{Wind}
\DeclareMathOperator{\cov}{cov}

\DeclareMathOperator{\ind}{ind}

\usepackage{amsthm}
\usepackage{thmtools}
\declaretheoremstyle[notefont=\bfseries,notebraces={}{},%
    headpunct={},postheadspace=1em]{mystyle}
\declaretheorem[style=mystyle,numbered=no,name=Theorem]{thm-hand}

\declaretheorem[style=plain, parent=section]{theorem}
\declaretheorem[sibling=theorem]{definition}
\declaretheorem[sibling=theorem]{lemma}

\declaretheorem[sibling=theorem]{proposition}

\declaretheorem[sibling=theorem, style=remark]{remark}
\declaretheorem[sibling=remark]{example}

\begingroup\expandafter\expandafter\expandafter\endgroup
\expandafter\ifx\csname pdfsuppresswarningpagegroup\endcsname\relax
\else
\fi

\definecolor{amaranth}{rgb}{0.9, 0.17, 0.31} 
\definecolor{carrotorange}{rgb}{0.80, 0.5, 0.01} 
\definecolor{citrine}{rgb}{0.89, 0.82, 0.04} 
\definecolor{dartmouthgreen}{rgb}{0.05, 0.5, 0.06} 
\definecolor{ballblue}{rgb}{0.13, 0.67, 0.8} 
\definecolor{ceruleanblue}{rgb}{0.16, 0.32, 0.75} 
\definecolor{amethyst}{rgb}{0.6, 0.4, 0.8} 
\definecolor{amber}{rgb}{1.0, 0.75, 0.0} 
\definecolor{burlywood}{rgb}{0.87, 0.72, 0.53} 

\usepackage{csquotes}
\usepackage[
  backend=biber,
  hyperref=true,
  backref=true,
  isbn=false,
  url=false,
  doi=false,
  natbib=true,
  eprint=true,
  useprefix=true,
  maxcitenames=99,
  maxbibnames=99,  
  maxalphanames=99, 
  minalphanames=99,
  safeinputenc,
  style=alphabetic,
  citestyle=alphabetic,
  block=space,
  datamodel=preamble/ext-eprint,
  sorting=nyt
]{biblatex}
\begin{document}


\title{A vanishing theorem in Siefring's intersection theory}
\author[N\, Manikandan]{Naageswaran Manikandan}
\address{Max Planck Institute for Mathematics in Bonn}
\email{manikandan@mpim-bonn.mpg.de}


\date{} 


\begin{abstract}
In 2009, R. Siefring introduced a homotopy-invariant generalized intersection number and singularity index for punctured pseudoholomorphic curves, by adding contributions from the curve's asymptotic behavior to the standard intersection number and singularity index.
In this article, we provide a stratification of the moduli space that describes the rate of asymptotic convergence of the pseudoholomorphic curves. 
Using this stratification, we provide a more intricate characterization of the curves for which these added contributions to the intersection number and singularity index vanishes. In doing so, we prove that the asymptotic contribution to intersection number and singularity index vanishes under generic perturbations. This means that in generic situations we only need to consider the usual intersections of the curves.
\end{abstract}

\makeatletter
\@namedef{subjclassname@2020}{%
  \textup{2020} Mathematics Subject Classification}
\makeatother


\maketitle

\section{Introduction}

The theory of pseudoholomorphic curves with punctures, initially introduced by H.~Hofer in \cite{Hofer_WeinsteinConjecture}, has become a crucial tool in symplectic topology. The local characteristics of punctured pseudoholomorphic curves are similar to those of closed curves. For instance, positivity of intersection, as discussed in \cite{McDuff_Positivity}, extends to punctured pseudoholomorphic curves. However, extending global estimates or inequalities to punctured pseudoholomorphic curves is challenging due to the subtlety and complexity involved. This includes not only extending the global estimates themselves but also adapting the definitions of the relevant quantities to punctured pseudoholomorphic curves. Issues arise particularly when these curves possess multiple ends converging to coverings of the same orbit or when examining self-intersection scenarios where an end approaches a multiple cover of an orbit.\\

R.~Siefring extensively studied such situations and defined `asymptotic intersection indices' in \cite{Siefinter2009}. These indices consider the complexities arising from multiple ends and covers, playing an important role in defining generalized intersection number and singularity index for punctured pseudoholomorphic curves. The generalized intersection number and singularity index are defined by adding suitably the contributions from the asymptotic indices to the usual intersection and singularity index. Notably, these asymptotic intersection indices rely solely on the behavior of the pseudoholomorphic curves in the vicinity of the punctures, hence the prefix `asymptotic'.\\

In this paper, we formulate and prove several theorems which, although known to experts and regarded as part of the mathematical folklore, have not previously appeared with proofs in the literature. We construct a stratification of the moduli space according to the asymptotic behavior of its elements, and use this structure to show that the asymptotic contributions to the intersection number and the singularity index vanish for generic choices of almost complex structures used in defining the moduli space. The expectation for such a result originates from Hutching's earlier works, such as \cite{Hutching_Indexinequality, Hutching_ECHindex}, which predate Siefring's formalization of intersection theory of punctured pseudoholomorphic curves in \cite{Siefinter2009}. This theorem can be perceived as an extension of generic vanishing of certain asymptotic contributions to the sum relative intersection pairing $Q_{\tau}$ and asymptotic writhe $w_{\tau}$ in Embedded Contact Homology (ECH).

Given a contact 3-manifold $(M,\xi=\ker(\alpha))$, the symplectization is the symplectic manifold $(\mathbb{R}\times M, d(e^r{\cdot}\alpha))$. Denote the moduli space of simple punctured pseudoholomorphic curves mapping into this symplectization from a fixed genus $g$ smooth surface, asymptotic to a fixed set of non-degenerate Reeb orbits in a fixed homology class by $\mathcal{M}^*(M,J)$. Now, we present a simplified version of the main theorem, omitting details. For precise definitions of the moduli space of simple curves and complete statements of the theorems, see Section~\ref{Moduli space of pseudoholomorphic curves} and Theorems~\ref{maintheorem1}, \ref{maintheorem2}, \ref{maintheorem3}, and \ref{maintheorem4}.

\begin{theorem}
There is a comeagre subset $\cJ_{\mathrm{reg}}$ of almost complex structures such that for every $J \in \cJ_{\mathrm{reg}}$ the subspace $\cM_{vanish}$ of $\cM^*(M,J)$ consisting of those curves $\tilde{u}$ whose asymptotic contribution to the singularity index $\mathrm{sing}(\tilde{u})$ vanishes contains an an open dense subset $\cO \subset \cM_{vanish}$.
\end{theorem}

\begin{theorem}
There is a comeagre subset $\cJ_{\mathrm{reg}}$ of almost complex structures such that for every $J \in \cJ_{\mathrm{reg}}$ the subspace $\cM_{vanish}$ of $\cM^*(M,J) \times \cM^*(M,J)$ consisting of the curves $(\tilde{u},\tilde{v})$ whose asymptotic contribution to the intersection $\tilde{u}*\tilde{v}$ vanishes contains an an open dense subset $\cO \subset \cM_{vanish}$.
\end{theorem} 

A corresponding theorem applies to simple curves into completions of symplectic cobordisms. Although the theorems above are stated specifically for contact 3-manifolds, the precise statements in Theorems~\ref{maintheorem1}, \ref{maintheorem2}, \ref{maintheorem3}, and \ref{maintheorem4} are formulated, more generally, for stable Hamiltonian structures on 3-manifolds, see Section~\ref{subsec: Stable Hamiltonian structures}.\\

There are several evident and straightforward implications of such a theorem. For instance, the adjunction formula provided in \cite[Theorem~4.6]{Siefinter2009} undergoes significant simplification, allowing the substitution of the generalised intersection number and the singularity index with their conventional counterparts for generic almost complex structures.\\

An important application of the ideas presented here is the development of an SFT-like theory that focuses exclusively on curves $u$ with vanishing generalized self-intersection number, i.e., $u * u = 0$. These curves were studied in detail by J.~Fish and R.~Siefring \cite{Fish_Siefring_FEF1}, who examined their index behavior in detail. The central motivation is to demonstrate that such curves generate a sub-$IBL_\infty$ algebra within the broader SFT algebra framework established in \cite{CFL_homotopyIBL}. This perspective is supported by the way the SFT compactness theorem \cite{SFTCompactness} interacts with the self-intersection number. The vanishing theorem established in this work is expected to play a crucial role in simplifying the moduli space well-definition problem, effectively reducing it to scenarios that have been extensively studied in embedded contact homology (ECH). Ultimately, the goal is for this theory to extend ECH by incorporating more intricate, higher-order algebraic structures.\\

The key step in establishing the result is to construct a stratification of the moduli space of simple punctured pseudoholomorphic curves organized according to the asymptotic decay of the curves, as described in Section~\ref{section: Stratification of the moduli space}. Once this stratification is achieved, it facilitates a comprehensive description of the subset of curves whose asymptotic index vanishes, as detailed in Section~\ref{section: Subsets of positive asymptotic intersection index}.
\subsection*{Outline}
In Section~\ref{section: background on Reeb dynamics}, we review the definitions and properties of stable Hamiltonian structures, Reeb orbits, and the associated asymptotic operators. In Section~\ref{section: pseudoholomorphic curves}, we introduce the moduli space of punctured pseudoholomorphic curves, which plays a key role in the formulation of the main theorem. Additionally, we provide a brief overview of the general construction, highlighting the essential elements required for the proof. In Section~\ref{section: intersection theory}, we present an overview of the algebraic formulation of Siefring's intersection theory. The new contributions begin in Section~\ref{section: Stratification of the moduli space}, where we introduce a stratification of the moduli space to characterize the asymptotic rate of convergence. Finally, in Section~\ref{section: Subsets of positive asymptotic intersection index}, we present the proofs of the main theorems.
\subsection*{Acknowledgement}
I began this work during my master's studies, and a portion of the results were previously presented in the Master thesis \cite{MasterThesis} submitted to Humboldt Universit\"at zu Berlin. I would like to thank my advisor Prof. Chris Wendl for proposing the research problem and engaging in numerous valuable discussions. I was funded by the Deutsche Forschungsgemeinschaft (DFG, German Research Foundation) under Germany's Excellence Strategy – The Berlin Mathematics Research Center MATH+ (EXC-2046/1, project ID: 390685689). I would like to thank the Max-Planck Institute for Mathematics in Bonn for its hospitality and financial support. I would like to also thank Gerard B. Gómez, Marc Kegel and Richard Siefring for reviewing earlier drafts of the paper, which significantly improved its readability. I would like to thank the reviewers for their careful reading of the manuscript and for their valuable suggestions, which have significantly improved the readability of the paper.

\section{Background on Reeb dynamics} \label{section: background on Reeb dynamics}
In this section, we present a slight generalization of contact structures on 3-manifolds called the stable Hamiltonian structures and review properties of asymptotic differential operators associated to the Reeb orbits of the stable Hamiltonian structures. These properties are extensively used when studying finite energy punctured pseudoholomorphic curves. Standard references for these concepts include \cite{SFTwendl} and \cite{Siefinter2009}.

\subsection{Stable Hamiltonian structures}\label{subsec: Stable Hamiltonian structures}
\begin{definition}\label{hamiltonian_structure}
Let $M$ be a closed oriented 3-manifold, \textit{a stable Hamiltonian structure} on $M$ is a pair $H = (\lambda, \omega)$ where $\lambda$ is a 1-form and $\omega$ is a 2-form on $M$ such that 
\begin{enumerate} 
    \itemsep0.2em
    \item $\lambda \wedge \omega$ is a volume form on $M$,
    \item $\omega$ is closed and 
    \item $\ker(\omega)\subset \ker(d\lambda)$.
\end{enumerate}
\end{definition}

From the definition, it follows that $\omega$ must have a rank of 2 everywhere and non-degenerate on hyperplane distribution defined by $\xi^{H} \coloneqq \ker \lambda$. It also determines a line bundle defined by $$l_{\omega} = \bigcup\limits_{p\in M} (p, \ker \omega_p).$$
Consequently, a stable Hamiltonian structure $H = (\lambda, \omega)$ leads to a splitting
\begin{equation*}\label{splitting_hamiltonian}
TM = l_{\omega} \oplus \xi^H
\end{equation*}
of the tangent bundle of $M$ into a trivial line bundle $l_{\omega}$ and a symplectic vector bundle $(\xi^{H}, \omega)$. Now, we define the \textit{Reeb vector field} $X_H$ as the unique section of $l_{\omega}$ satisfying $\lambda(X_{H}) = 1$.

\begin{example}
Given a 3-manifold $M$ and a contact structure $\xi=\ker (\alpha)$, the pair $(\alpha,d\alpha)$ defines a stable Hamiltonian structure on $M$. 
\end{example}

\begin{example}\cite{HZ-Bookdynamics}\label{example: stabilizing vector field}
Given a compact hypersurface $M$ in a symplectic manifold $(W,\omega)$, a transverse vector field $Z$, called the \textit{stabilizing vector field}, on a neighborhood of $M$ satisfying
$$\ker \left((\Phi^t_Z)^*\omega|_{TM}\right)=\ker (\omega|_{TM}) $$
determines a stable Hamiltonian structure, where $\Phi^t_Z$ is the flow determined by the vector field $Z$. The pair $(\lambda:=i_{Z}\omega,\omega)$ restricted to $TM$ determines a stable Hamiltonian structure.
\end{example}

The periodic orbits of $X_H$, called the Reeb orbits, play a role similar to the Reeb orbits in contact manifolds. There is a natural self-adjoint first order differential operator associated to each Reeb orbit called the asymptotic operator whose spectral properties prove essential when giving explicit description of the pseudoholomorphic curves near the punctures. In the next subsection, we will review the definition and some properties of these operators.

\subsection{Asymptotic operators}\label{sec: Asymptotic operator}

Given $\tau>0$, a map $\gamma \in C^{\infty}(S^1 , M)$ is called a $\tau$-periodic Reeb orbit if $\gamma$ satisfies the equation $$d\gamma(t)(\partial_t) = \tau \cdot X_H (\gamma(t))$$ for all $t \in S^1$. Given a symmetric connection $\nabla$ and a compatible complex structure $J$ on the symplectic vector bundle $(\xi^{H},\omega)$, the differential operator 

\[\begin{array}{cccc}
      A_{\gamma,J}\colon & C^{\infty} (\gamma^*\xi^H) &\rightarrow &C^{\infty} (\gamma^*\xi^H)\\
       & \eta &\mapsto &-J(\nabla_t \eta - \tau \nabla_{\eta} X_H) \\
\end{array}\]
associated to each $\tau$-periodic Reeb orbit $\gamma$ is called the \textit{asymptotic operator} associated to $\gamma$. The asymptotic operator $A_{\gamma,J}$ is independent of the choice of the symmetric connection and is symmetric with respect to the $L^2$-inner product on $C^{\infty}(\gamma^*\xi^H)$ given by 
$$(\eta_1,\eta_2)\coloneqq \int_{S^{1}}\omega\left(\eta_1(t),J\eta_2(t)\right) dt.$$

Given a simple $\tau$-periodic orbit $\gamma$, we define the multiply covered periodic orbit of multiplicity $m>0$, denoted by $\gamma^m$, by $\gamma^m(t) \coloneq \gamma(mt)$. 

Given a compatible complex structure $J$ on $(\xi^{H},\omega)$, the unitary line bundle $(\gamma^{*}\xi^{H},\omega,J)$ over $S^1$ is trivializable. For multiply covered orbits, we always choose trivializations of $(\gamma^{m})^{*}\xi^{H}$ obtained as pullbacks of trivializations of $\gamma^{*}\xi^{H}$ under the usual covering map.

A $\tau$-periodic orbit $\gamma$ is called \textit{non-degenerate}  if the associated asymptotic operator has a trivial kernel, i.e., $\ker(A_{\gamma,J})=0$. A stable Hamiltonian structure $(\lambda,\omega)$ on $M$ is said to non-degenerate if all periodic orbits are non-degenerate.\\

For functional analytic purposes we consider the extension of asymptotic operators to $L^2(\gamma^*\xi^{H})$ with a dense domain $H^1(\gamma^*\xi^{H})$, which makes them unbounded self-adjoint operators. The following lemma provides a comprehensive description of their spectrum and associated winding numbers, a crucial aspect in defining the generalized intersection of punctured pseudoholomorphic curves.
\begin{lemma} \cite[Section~3]{HWZ2}\label{asymptotic_dimension}
Let $\gamma$ be a simple periodic Reeb orbit, let $T(\gamma^* \xi^H)$ denote the homotopy classes of unitary trivializations of $(\gamma^* \xi^H , \omega, J)$ and let $A_{\gamma^m,J}$ be the asymptotic operator associated to $\gamma^m$ for $m>0$. Then there exists a map $w\colon \sigma(A_{\gamma^m ,J}) \times T(\gamma^* \xi^H ) \rightarrow \mathbb{R}$ which satisfies

\begin{enumerate}[topsep=5pt,itemsep=5pt]
    \item If $v \in H^1((\gamma^m)^*\xi^{H})$ is an eigenvector of $A_{\gamma^m,J} (v)$ with eigenvalue $\lambda$, i.e., $A_{\gamma^m,J} (v) = \lambda \cdot v$, then $w(\lambda, [\Phi]) = \wind(\Phi^{-1} v)$.
    \item The map $w(\lambda, [\Phi])$ is non-decreasing in eigenvalues.
    \item If $m(\lambda) = \dim \ker(A_{\gamma^m,J} - \lambda)$ denotes the multiplicity of $\lambda$ as an eigenvalue we have for every $k \in \mathbb{Z}$ and $[\Phi] \in T(\gamma^* \xi^H)$ that $$\sum_{\{\lambda \mid w(\lambda,[\Phi])\}=k} m(\lambda) = 2.$$

\end{enumerate}
\end{lemma}

Following the definition of the winding number, we will define several quantities  using the description of the spectrum given above. Given a simple periodic orbit $\gamma$, a unitary trivialization $\Phi \in T(\gamma^* \xi^H)$, we define:

\begin{enumerate}
    \itemsep0.4em
    \item $\sigma_{\max}^{-}(\gamma^m)$ as the largest negative eigenvalue, i.e., $$\sigma_{\max}^{-}(\gamma^m)\coloneqq \max(\sigma(A_{\gamma^m,J}) \cap \mathbb{R}^{-}),$$ 
    \item $\alpha^{\Phi}(\gamma^m)$ as the winding number of the largest negative eigenvalue $\sigma_{\max}^{-}(\gamma^m)$, i.e., $$\alpha^{\Phi}(\gamma^m)\coloneqq \wind (\sigma_{\max}^{-}(\gamma^m); \Phi),$$ 
    \item $\overline{\sigma}(\gamma^m)$ as the covering number of the largest negative eigenvalue, i.e., $$\overline{\sigma}(\gamma^m )=\cov(\sigma_{\max}^{-}(\gamma^m)) = \gcd(m, \alpha^{\Phi}(\gamma^m)).$$ The covering number $\overline{\sigma}(\gamma^m)$ does not depend on the trivialization.
    \item the parity $p(\gamma^m)$ of $\gamma^m$ by
    \[  p(\gamma^m)= \left\{
    \begin{array}{ll}
        0 & \mathrm{if\ }\exists \mu \in \sigma(A_{\gamma^m,J}) \cap \mathbb{R}^{+} \mathrm{\ with\ } \wind (\mu, \Phi) = \alpha^{\Phi}(\gamma^m)\\
        1 & \mathrm{otherwise.} \\
    \end{array} 
\right. \]
    \item the Conley--Zehnder index $\mu^{\Phi}(\gamma^m)$ of $\gamma^m$ by
$$\mu^{\Phi}(\gamma^m)=2\alpha^{\Phi}(\gamma^m)+p(\gamma^m).$$
\end{enumerate}

We will at times suppress the choice of trivialization in our notation for $\alpha$ or $\mu$, but it should always be understood that a choice of trivialization is necessary to define these quantities. Even though the asymptotic operator $A_{\gamma^m,J}$ depends on a choice of $J$ the Conley--Zehnder index of an orbit is independent of this choice. This definition of the Conley--Zehnder index is suitable only for stable Hamiltonian structures on 3-manifolds. A suitable generalisation of stable Hamiltonian structures to higher dimension can be found in the book \cite[Section~6.1]{SFTwendl} and the generalisation of the Conley--Zehnder index Reeb orbits in higher dimensional Hamiltonian structures can be found in \cite[Section~3.4]{SFTwendl}. 

\section{Background on pseudoholomorphic curves} \label{section: pseudoholomorphic curves}
In this section, we recall the definitions of symplectization and completed symplectic cobordisms of stable Hamiltonian structures and finite energy pseudoholomorphic curves into them. The primary reference for this discussion is \cite{SFTwendl}.

\subsection{Pseudoholomorphic curves in symplectization} \label{subsection: Pseudoholomorphic curves in symplectization}
Given a contact manifold $(M,\xi=\ker(\alpha))$, the symplectization is the symplectic manifold $(\mathbb{R}\times M, d(e^r{\cdot}\alpha))$. When $M$ is equipped with a stable Hamiltonian structure, then $\mathbb{R}\times M$ does not have a canonical symplectic structure but a family of symplectic structures as described below.\\

Fix $\epsilon>0$ small and define
$$\cT :=\{ \phi \in C^{\infty}(\mathbb{R},(-\epsilon,\epsilon))\mid \phi' >0\}.$$

If $\epsilon >0$ is small enough, then $\omega+ d(r {\cdot} \alpha)$ is symplectic on $(-\epsilon,\epsilon)\times M$. When this symplectic form is pulled back via the level preserving embedding $\mathbb{R}\times M \rightarrow (-\epsilon,\epsilon)\times M$ determined by $\phi \in \cT$ gives rise to a symplectic form $$\omega_{\phi} = \omega + d (\phi(r){\cdot} \alpha)$$ on $\mathbb{R}\times M$. Now, we need to define a suitable class of tame and compatible almost complex structures on $(\mathbb{R}\times M,\omega_{\phi})$ as the typical definitions do not consider the natural symmetries of tangent bundle of $\mathbb{R}\times M$. A compatible almost complex structures $J$ on the symplectic vector bundle $(\xi^H , \omega)$ can be extended to an $\mathbb{R}$-invariant almost complex structure $\tilde{J}$ on $\mathbb{R}\times M$ by requiring
$$\tilde{J}(\partial_r) = X_H$$ 
$$\tilde{J}(X_H)=-\partial_r$$
where $r$ is the parameter along $\mathbb{R}$. Such an $\mathbb{R}$-invariant almost complex structure on $\mathbb{R}\times M$ is called the standard \textit{cylindrical almost complex structure}. The space of standard cylindrical almost complex structures will also be denoted by $\cJ(M,H)$. If $\epsilon > 0$ is sufficiently small, it is routine to verify that for any triple $(\mathbb{R}\times M, \omega_{\phi}, \tilde{J})$, the form $\omega_{\phi}$ tames $\tilde{J}$ for every $J \in \mathcal{J}(M,H)$ and $\phi \in \mathcal{T}$.

\begin{definition}
Given a connected Riemann surface $(\Sigma,j)$, a finite set $\Gamma \subset \Sigma$ and $J \in \cJ(M,H)$, a smooth map
$$\tilde{u} = (a, u)\colon (\Sigma \setminus \Gamma,j) \rightarrow (\mathbb{R} \times M,\tilde{J})$$
is called punctured pseudoholomorphic curve if the differential $d\tilde{u}$ is a complex linear map with respect to $j$ and cylindrical almost complex structure $\tilde{J}$, i.e.,
$$\tilde{J} \circ d\tilde{u} = d\tilde{u}\circ j \left(\iff d\tilde{u} + \tilde{J} \circ d \tilde{u} \circ j = 0\right).$$ 
\end{definition}

\begin{definition}
The energy of a punctured pseudoholomorphic curve $\tilde{u}\colon (\Sigma \setminus \Gamma,j) \rightarrow (\mathbb{R} \times M,\tilde{J})$  is defined as
$$E\coloneqq \sup_{\phi \in \cT} \int_{{\Sigma\setminus \Gamma}} u^*\omega_{\phi}.$$
\end{definition}

As the triple $(\mathbb{R}\times M, \omega_{\phi},\tilde{J})$ is taming for every  $J\in \cJ(M,H)$ and $\phi \in \cT$, the energy is always
non-negative, and is strictly positive unless the curve is constant. Similar to other Floer-type theories, restricting to finite energy curves yields the following three possibilities near each puncture  $z_0\in \Gamma$.

\begin{enumerate}
\itemsep0.4em
\item Removable punctures: The map $\tilde{u} = (a, u)$ is bounded near $z_0$, in which case $\tilde{u}$ admits a smooth, $\tilde{J}$-holomorphic extension over the puncture.
    
\item Positive punctures: The function $a$ is bounded from below near $z_0$ but not from above. In this scenario, we have the following:
\vspace{0.4em}
\begin{enumerate}
    \itemsep0.4em
    \item There exists a non-degenerate $\tau$-periodic Reeb orbit $\gamma$,
    \item there exists an embedding $\psi\colon [R, \infty) \times S^1 \rightarrow \Sigma \setminus \Gamma$ into a punctured neighborhood of $z_0$, with the property that $\lim\limits_{s\rightarrow \infty} \psi(s, t) = z_0$, and
    \item a map $U\colon [R, \infty) \times S^1 \rightarrow \gamma^* \xi^H$, where $U(s, t) \in \xi_{\gamma(t)}^H$ for all $(s, t) \in [R, \infty) \times S^1$ satisfying the following conditions.
    \vspace{0.4em}
    \begin{enumerate}
    \itemsep0.4em
        \item The equation $\tilde{u}(\psi(s, t)) = (\tau s, \exp_{\gamma(t)} U(s, t))$ holds and
        \item the limit $U(s, \cdot) \rightarrow 0$ as $s\rightarrow \infty$ in the $C^\infty$ topology.
    \end{enumerate}
\end{enumerate}

\item  Negative punctures: The function $a$ is bounded from above near $z_0$ but not from below. In this scenario, we have the following:
\vspace{0.4em}
\begin{enumerate}
    \itemsep0.4em
    \item There exists a non-degenerate $\tau$-periodic Reeb orbit $\gamma$,
    \item there exists an embedding $\psi\colon (-\infty,-R] \times S^1 \rightarrow \Sigma \setminus \Gamma$ into a punctured neighborhood of $z_0$, with the property that $\lim\limits_{s\rightarrow -\infty} \psi(s, t) = z_0$, and
    \item a map $U\colon (-\infty,-R] \times S^1 \rightarrow \gamma^* \xi^H$, where $U(s, t) \in \xi_{\gamma(t)}^H$ for all $(s, t) \in (-\infty,-R] \times S^1$ satisfying the following conditions.
    \vspace{0.4em}
    \begin{enumerate}
    \itemsep0.4em
        \item The equation $\tilde{u}(\psi(s, t)) = (\tau s, \exp_{\gamma(t)} U(s, t))$ holds and
        \item the limit $U(s, \cdot) \rightarrow 0$ as $s\rightarrow -\infty$ in the $C^\infty$ topology.
    \end{enumerate}
\end{enumerate}
\end{enumerate}
Where $\exp$ denotes the exponential map of the metric $$g_{H,J} \coloneqq \lambda \otimes \lambda + \omega(\cdot, J{\cdot})$$ on $M$. We will call a pair $(U, \psi)$ satisfying the conditions above an \textit{asymptotic representative} of $[\Sigma, j, z_0 , \tilde{u}]$ and it is clear from the $C^{\infty}$ convergence of $\tilde{u}$ uniquely determine $(U, \psi)$ up to restriction of the domain and an `asymptotic decoration', which essentially provides a specific parametrization for the Reeb orbit $\gamma$.\\

We will henceforth assume that all punctured pseudoholomorphic curves have finite energy and removable punctures have been removed.
Thus that all punctures are either positive or negative punctures at which the curves in question are asymptotic to Reeb orbits.
Hence the terms `finite energy' and `asymptotically cylindrical' pseudoholomorphic curves are used as synonyms.

\subsection{Pseudoholomorphic curves in  completed cobordisms}
Consider a symplectic cobordism denoted as $(W, \omega)$, which connects $(M^-, H^-)$ to $(M^+, H^+)$, meaning that $\partial W = M^+ - M^-$. Within this definition, we have the existence of a stabilizing vector field $Z$ that directs inward at $M^-$ and outward at $M^+$, inducing the given stable Hamiltonian structures on $M^{\pm}$, see Example~\ref{example: stabilizing vector field}.\\

For a sufficiently small $\epsilon > 0$, we define  
$$\cT_0 \coloneqq \{ \phi \in C^{\infty}(\mathbb{R}, (-\epsilon,\epsilon) \mid \phi ' >0 \mathrm{\ and\ } \phi(r)=r \mathrm{\ for\ } r \mathrm{\ near\ } 0 \}$$ and construct a completed cobordism $\widehat{W}$ by attaching symplectic ends $(-\infty, 0] \times M^-$ and $[0, \infty) \times M^+$ to the negative and positive ends, respectively, for every $\phi \in \cT_0$.\\

According to the symplectic neighborhood theorem (see \cite[Section~6.2]{SFTwendl}), the symplectic structure defined on $(-\infty, 0] \times M^-$ and $[0, \infty) \times M^+$ in Section~\ref{subsection: Pseudoholomorphic curves in symplectization} for every $\phi \in \cT_0$ can be patched together with $\omega$ on $W$ to form a coherent global structure $\omega_{\phi}$ on the completion. We define the set $\cJ_{\tau}(\widehat{W}, \omega_{\phi})$ as the collection of $\omega_{\phi}$-tamed almost complex structures on $\widehat{W}$ that coincide with an element of $\cJ(M^+,H^+)$ (resp. $\cJ(M^-,H^-)$) on the regions $[0, \infty) \times M^+$ (resp. $(-\infty, 0] \times M^-$). We denote the subspace of $\omega_{\phi}$-compatible almost complex structures satisfying these constraints as $\cJ(\widehat{W},\omega_{\phi})$.\\

A notion of energy for pseudoholomorphic curves into completed cobordisms can be defined. Much like in the case of symplectization, pseudoholomorphic curves with finite energy exhibit a characteristic feature: each non-removable puncture corresponds to either a positive or negative puncture, where the curve is positively or negatively asymptotic to trivial cylinders over periodic orbits.

\subsection{Moduli space of pseudoholomorphic curves} \label{Moduli space of pseudoholomorphic curves}
In this section, we give a quick overview of the moduli space of punctured pseudoholomorphic curves into completed cobordisms and symplectizations.\\
\vspace{1pt}

We consider two asymptotically cylindrical pseudoholomorphic curves equivalent if they are related to each other by biholomorphic maps of their domains that preserve the order of punctures. These  equivalence classes are called unparametrized pseudoholomorphic curves. \\
When we refer to moduli spaces, we will always mean the space of unparametrized pseudoholomorphic curves. The topology on this space is such that a sequence  is considered to converge if and only if one can find parametrizations with a fixed punctured domain $\Sigma\setminus \Gamma$ such that the complex structures on $\Sigma$ converge in $C^{\infty}$ while the pseudoholomorphic maps $\Sigma\setminus \Gamma \rightarrow \widehat{W}$ converge in $C^{\infty}$ on compact subsets and in $C^0$ near the cylindrical ends with respect to a translation-invariant metric on the
ends.\\

Within this moduli space, there exist multiple components, each having different dimensions. In order to produce a useful description , we must restrict to a subset of pseudoholomorphic curves over a fixed genus $g$ smooth surface, asymptotic to a fixed set of non-degenerate Reeb orbits $\gamma_{\pm}=\{\gamma_z\}_{z \in \Gamma_{\pm}}$ in $(M^{\pm},H^{\pm})$ in a fixed homology class $A\in H_2(W,\gamma_+\cup \gamma_-)$. We denote this subset of the moduli space simply by $\cM_g(\widehat{W},J)$, keeping in mind the rest of the data needed to define this. This moduli space could still be disconnected, but under appropriate transversality conditions, it is possible to provide a dimension formula.\\

The moduli space $\cM_g(\widehat{W},J)$ can be interpreted abstractly as the zero set of a section $\cS$ of the Banach space vector bundle $\cE^{k-1,p,\delta} \rightarrow T \times \cB^{k,p,\delta}$, for sufficiently small $\delta$. Here, $\cB^{k,p,\delta}\coloneqq W^{k,p,\delta}(\dot{\Sigma},\widehat{W})$ denotes the Banach manifold of continuous maps from a punctured genus $g$ surface $\dot{\Sigma}:=\Sigma\setminus \Gamma$ to $\widehat{W}$ converging to Reeb orbits $\gamma_{\pm}$ exponentially fast (at least of order $\delta$) near the punctures and whose $k$-th derivatives are of class $L^p$ and $T$ denotes a sufficiently small neighbourhood in the Teichm\"uller space $\cT$ punctured genus $g$ surface $\dot{\Sigma}$. The tangent space of $\cB^{k,p,\delta}$ at $\tilde{u} \in \cB^{k,p,\delta}$ being
$$T_{\tilde{u}}\cB^{k,p,\delta} = W^{k,p,\delta}(\dot{\Sigma},\tilde{u}^*T\widehat{W}) \oplus V_{\Gamma}$$
where $V_{\Gamma}\subset \Gamma(\tilde{u}^*T\widehat{W})$ represents a non-canonical vector space of dimension $2|\Gamma|$, which accounts for automorphisms in $\mathbb{R} \times S^1$.\\

The space $\cE^{k-1,p,\delta}$ denotes the vector bundle whose fibre at $(j,\tilde{u})\in T \times \cB^{k,p,\delta}$ is \\
$W^{k-1,p,\delta}(\dot{\Sigma}, \bigwedge^{0,1} T^*\dot{\Sigma}\otimes \tilde{u}^*T\widehat{W})$ and the section
$ \cS\colon T \times \cB^{k,p,\delta} \rightarrow \cE^{k-1,p,\delta}$ is given by 
$$(j, \tilde{u}) \mapsto (\tilde{u}, \overline{\partial}_{J,j}(\tilde{u}))$$
where $\overline{\partial}_{J,j}$ is the Cauchy--Riemman operator $\overline{\partial}_{J,j}(\tilde{u})=d\tilde{u} + \tilde{J} \circ d \tilde{u} \circ j$. It suffices to consider exponentially decaying functions, since every pseudoholomorphic curve decays exponentially near the puncture \cite{HWZ1,HWZ2,HWZ3}; see also Theorem~\ref{asymp_rep1}. For further details on exponentially weighted moduli spaces, see \cite[Section~7.2]{SFTwendl}. The section $\cS$ is Fredholm of index 
$$\ind(\tilde{u})=(n-3)\chi(\dot{\Sigma})+2c_1(A,\tau)+\sum_{z\in \Gamma_{+}}\mu_{CZ}^{\tau}(\gamma_z)-\sum_{\gamma \in \Gamma_{-}}\mu_{CZ}^{\tau}(\gamma_z).$$

If $\cS$ is transverse to the zero-section, then the moduli space is a smooth orbifold of above dimension. Heuristically, we're dealing with an orbifold rather than a simple manifold due to two primary reasons: firstly, $\cM_g(\widehat{W},J)$ contains curves with non-trivial automorphisms. Secondly, the moduli space of complex structures on a smooth genus $g$ surface is an orbifold obtained as a quotient of the Teichm\"uller space.\\

Transversality does not hold in general, we can achieve it by either by considering only a subset of pseudoholomorphic curves or by introducing suitable perturbations of almost complex structures on $\widehat{W}$ or both. A curve $\tilde{u}$ is said to be \textit{Fredholm regular} if $\cS$ intersects the zero section transversely at $\tilde{u}$: in this case a neighborhood of $\tilde{u}$ is a smooth orbifold of dimension $\ind(\tilde{u})$. Thus the space 
$$\cM_g^{\mathrm{reg}}(\widehat{W},J) \subset \cM_g(\widehat{W},J)$$ consisting of Fredholm regular curves is a smooth orbifold of dimension $\ind(\tilde{u})$.\\

We have outlined the structure of the moduli space of pseudoholomorphic curves into completed cobordisms. This discussion can be extended to symplectizations, while keeping the description of the moduli space and the dimension formula unchanged. Finally, we define the moduli spaces essential for presenting the folklore theorem mentioned in the introduction.\\

Given an open subset $O \subset W$ of the cobordism, we define the moduli space $\cM^*(\widehat{W},J)$ as the subset of curves $\tilde{u} \in \cM_g(\widehat{W},J)$ that have an injective point mapped into $O$, i.e,
\begin{equation*}
\cM^*(\widehat{W},J;O) = \cM^*(\widehat{W},J)\coloneqq \left\{ \tilde{u}\in \cM_g(\widehat{W},J) \, \middle|
\begin{array}{c}
\exists \mathrm{\ an\ injective\ point\ } z\in \Sigma\setminus \Gamma \mathrm{\ such \ that\ }\\ \tilde{u}(z) \in O
\end{array}
\right\}.
\end{equation*}
\vspace{1pt}

We can obtain the transversality for curves in $\cM^*(\widehat{W},J)$ by suitably perturbing the almost complex structure. Thus for a generic choice of almost complex structure in $\cJ(\widehat{W},\omega_{\phi})$, the moduli space $\cM^*(\widehat{W},J)$ is a smooth manifold of dimension $\ind(\tilde{u})$.\\

Similarly, given an open subset $O \subset M$, we define the moduli space $\cM^*(M,J)$ as
 
\begin{equation*}
\cM^*(M,J;O) = \cM^*(M,J)\coloneqq \left\{ 
\tilde{u}\in \cM_g(\mathbb{R}\times M,J) \,\middle|\,
\begin{array}{c}
\exists \, \text{an injective point } z \in \Sigma \setminus \Gamma \text{ such that} \\
\tilde{u}(z) \in O \text{ and } \pi_{\xi^H} \circ d\tilde{u}(z) \neq 0
\end{array}
\right\}.
\end{equation*}
\vspace{1pt}

We can obtain the transversality for curves in $\cM^*(M,J)$ for a generic choice of cylindrical almost complex structure. Thus for a generic choice of almost complex structure, the moduli space $\cM^*(M,J)$ is also a smooth manifold of dimension $\ind(\tilde{u})$. We have suppressed the choice of the open subset $O$ in our notation for the moduli spaces $\cM^*(M,J)$ and $\cM^*(\widehat{W},J)$. However, it's crucial to remember that a choice of $O$ is necessary for defining these moduli spaces.

\section{Background on Siefring's intersection theory} \label{section: intersection theory}
In this section, we present a quick overview of Siefring's intersection theory of pseudoholomorphic curves \cite{Siefinter2009}. We will present an algebraic formulation based on the expansion of the asymptotic representative, as in Theorem~\ref{asymp_rep1}, instead of the geometric approach, as this is employed in later sections to prove the vanishing theorem, see Remark~\ref{rem:GeometricDecrpition}.

\subsection{Local Intersection theory}

In this subsection, we will introduce quantities that describe the behavior of pseudoholomorphic curves near the punctures. To accomplish this, we first need the following definition:\\

A pseudoholomorphic end model is a quadruple $(S,j,\tilde{u},z)$ where $(S,j)$ is a Riemann surface (not necessarily closed) without boundary, $z\in S$ is a point and $\tilde{u}:S \setminus \{z\} \rightarrow \mathbb{R}\times M$ is an asymptotically cylindrical pseudoholomorphic map. We say two such end models $(S,j,\tilde{u},z)$ and $(S',j',\tilde{v},w)$ are equivalent if there exists an open neighborhood  $\cZ\subset S$ containing $z$, and a holomorphic embedding $\psi : \cZ \rightarrow S'$ with $\psi(z)=w$ so that $\tilde{u}=\tilde{v}\circ \psi$ on $\cZ \setminus \{z\}$. An equivalence class $[S,j,\tilde{u},z]$ of pseudoholomorphic end models will be referred to as a \textit{pseudoholomorphic end}. We occasionally refer to a pseudoholomorphic end simply as $[\tilde{u}; z]$, as the domain does not play a significant role in the subsequent discussions. Given a pseudoholomorphic end denoted by $[\tilde{u}; z]$, we define its $m$-fold multiple cover as $m \cdot [\tilde{u}; z]$.\\

The following notation is fixed for the rest of this section: the pseudoholomorphic end $[S,j,\tilde{u},z]$ (resp. $[S',j',\tilde{v},w]$) is positively asymptotic to a non-degenerate Reeb orbit $\gamma_z^{m_z}$ (resp. $\gamma_w^{m_w}$) with asymptotic representative $(U, \psi)$ (resp $(V, \phi)$). Here, the curves \(\gamma_z\) and \(\gamma_w\) are assumed to be simple periodic Reeb orbits. Although we have not formally defined an asymptotic representative associated with a pseudoholomorphic end, the extension of the definition is straightforward. The definitions and the results given in this section also apply to a pair of pseudoholomorphic curves that are negatively asymptotic to a cylinder with minor modification, see \cite{Siefinter2009,Wendlcontact3mani2020}.\\

The self intersection problem alluded to in the opening paragraph of the paper, which arises when an end approaches a multiple cover of an orbit and when multiple ends of the same curve approaches the same orbit can be characterized using asymptotic self intersection indices $\delta_{\infty}(\tilde{u}; z)$ and $\delta_{\infty}([\tilde{u}; z], [\tilde{v}; w])$ respectively. The purpose of this section is to describe the features of these indices without getting into the details of defining these indices. These quantities are defined using the expansion of the asymptotic representative described below.\\

\begin{theorem}\label{asymp_rep1}\cite[Theorem~2.2]{Siefinter2008}
Given $\gamma:=\gamma_z^{m_z}=\gamma_w^{m_w}$, the difference of the asymptotic representatives $(U,\phi)$ and $(V,\psi)$ can be written as
$$U(s, t)- V(s, t) = e ^{\lambda s} \left[e(t) + r(s, t)\right]$$
where $e$ is an eigenvector of the asymptotic operator $A_{\gamma,J}$ with eigenvalue $\lambda < 0$ and $r$ satisfies
$$|\nabla^i_s \nabla^j_t r(s, t)| \leq M_{ij} e^{-ds}$$
for all $(s, t) \in [R, \infty ) \times S^1 , (i, j) \in \mathbb{N}^{2}$ , and some appropriate constants $d > 0$ , $M_{ij} > 0$.
\end{theorem}

A refined formulation of Theorem~\ref{asymp_rep1} involves a sum of exponential terms with decreasing eigenvalues in the exponents (see \cite[Theorem~2.3]{Siefinter2008}), namely
\begin{equation}\label{asymp_rep2}
    U(s,t) = \sum_{i=1}^N e^{\lambda_i s} \left[e_i(t) + r_i(s,t)\right],
\end{equation}
where $\lambda_j < \lambda_i$ for $j > i$, and each $e_i \neq 0$ is an eigenvector of $A_{\gamma,J}$ with eigenvalue $\lambda_i$. The functions $r_i$ satisfy $r_i(s,t) = r_i\!\left(s, t + \tfrac{1}{k_i}\right)$.
The sequence $\{k_i\}$ is strictly decreasing in $i$ and is defined inductively by setting $k_1 = \operatorname{cov}(e_1)$ and $k_i = \gcd(k_{i-1}, \operatorname{cov}(e_i))$. Moreover, the remainder terms $r_i$ satisfy $\left|\nabla_s^l \nabla_t^m r_i(s,t)\right| \leq M_{lm} e^{-ds}$.
See Section~\ref{sec: Asymptotic operator} for the definition of the covering number. 

\begin{remark}
This remark that follows from Theorem~\ref{asymp_rep1} describes the behavior of a pseudoholomorphic curve or a pair of curves near the punctures. The definition of the asymptotic self intersection indices implicitly depends on this remark.
\begin{enumerate}
    \item A pseudoholomorphic end $[\tilde{u};z]$ is either embedded or a multiple cover of an embedded curve.
    \item A pair of pseudoholomorphic ends $[\tilde{u};z]$ and $[\tilde{v};w]$ are either non-intersecting, or are equal, or one is a multiple of the other, or else both are multiple coverings of a common curve.
    \end{enumerate}
\end{remark}

We say that a pseudoholomorphic end $[S,j,\tilde{u},z]$ winds if $U(s,t) \neq 0$ on a sufficiently small neighbourhood of $z$. In that case, we define \textit{the asymptotic winding number} of $\tilde{u}$ at $z$ relative to the unitary trivialization $\Phi$ of $(\gamma^* \xi^H , \omega, J)$ as $$\wind^{\Phi}_{\infty}(\tilde{u}; z)\coloneqq \wind\left(\Phi^{-1} e (\tilde{u}; z)\right).$$
where $e(\tilde{u}; z)$ be the eigenvector in the asymptotic expression of $\tilde{u}$ at $z$ from Theorem~\ref{asymp_rep1}.\\

We say that a pseudoholomorphic end $[S,j,\tilde{u},z]$ is embedded if it admits a representative end model that is embedded. If $m_z \geq 2$, then $U(s, t)-U\left(s,t+ \frac{j}{m_z}\right) = 0$ if and only if $j$ is a multiple of $m_z$ and Theorem~\ref{asymp_rep1} provides the following description for the difference of asymptotic representatives,
$$U(s, t)- U\left(s, t + \frac{j}{m_z} \right) = e^{\mu_j s} \left[e'_j(t) + r'_j (s, t)\right]$$
where $e'_j \neq 0$ an eigenvector of $A = A_{\gamma^{m_z},J}$ with eigenvalue $\mu_j < 0$ and $r'_j$ converging exponentially to $0$. In that case, we define \textit{the secondary winding number} of $\tilde{u}$ at $z$ relative to the trivialization $\Phi$ as 
$$\wind^{\Phi}_2(\tilde{u}; z)= \sum_{j=1}^{m_z-1} \wind\left(\Phi^{-1} (e'_j(t)\right).$$

\begin{lemma} \label{delta1}\cite[Subsection~3.1]{Siefinter2009}
Let $[S,j,\tilde{u},z]$ be an embedded pseudoholomorphic end that winds, then the quantity
$$d_0(\tilde{u} ; z) \coloneqq \alpha^{\Phi}(\gamma_z^{m_z}) -\wind_{\infty}^{\Phi}(\tilde{u}; z)$$
is non-negative and independent of the choice of trivialization. Moreover, the following quantities
\[\begin{array}{ll}
\Delta_1 (\tilde{u}; z)\coloneqq& (m- 1)d_0(\tilde{u}; z)- \overline{\sigma}(\gamma_z^{m_z}) + \cov\left(e_1 (\tilde{u}; z)\right)\\
\Delta_2 (\tilde{u}; z) \coloneqq & (m-1) \wind^{\Phi}_{\infty}(\tilde{u}; z) - \cov\left(e_1 (\tilde{u}; z)\right) + 1 - \wind_2^{\Phi} (\tilde{u}; z)
\end{array}\]
are even, non-negative, and independent of trivialization. Here, the eigenvector $e_1(\tilde{u}; z)$ is given by the Equation~\ref{asymp_rep2}.
\end{lemma}

The asymptotic self-intersection indices defined below are built using the indices $\Delta_1$ and $\Delta_2$. Therefore, understanding the vanishing set of $\Delta_1$ and $\Delta_2$ is crucial for analyzing the vanishing set of the asymptotic self-intersection indices. Here is a lemma that describes exactly when the quantities $\Delta_1 (\tilde{u}; z)$ and $\Delta_2 (\tilde{u}; z)$ vanish. 

\begin{lemma}\cite[Section~3.1]{Siefinter2009} \label{self-intersction=0}
Let $[S,j,\tilde{u},z]$ be defined as above. Then, we have that $\Delta_1 (\tilde{u}; z) = 0$ if and only if at least one of the following holds:
\begin{itemize}[topsep=5pt,itemsep=5pt]
    \item $m_z=1$
    \item $d_0 (\tilde{u}; z) = 0$
    \item $d_0 (\tilde{u}; z) = 1 \mathrm{\ and\ }  \overline{\sigma}(\tilde{u}; z)= m_z.$
\end{itemize}
We have that $\Delta_2 (\tilde{u}; z) = 0$ if and only if the asymptotic representative of $\tilde{u}$ near $z$, as described in Equation~\ref{asymp_rep2}, has one term with $\cov(e_1 (\tilde{u}; z))=1$ or two terms with the winding of the eigenvectors appearing in this formula differ precisely by 1.
\end{lemma}

We now define \textit{the asymptotic self-intersection index} $\delta_{\infty}(\tilde{u}; z)$ of $[S,j,\tilde{u},z]$ to be
\begin{equation*}
    \begin{split}
        \delta_{\infty}(\tilde{u}; z) \coloneqq \frac{1}{2}\left(\Delta_{1} (\tilde{u}; z)+\Delta_2(\tilde{u}; z)\right).
    \end{split}
\end{equation*}

We close this section by defining the relative asymptotic winding number and the relative asymptotic self intersection index $\delta_{\infty}([\tilde{u}; z], [\tilde{v}; w])$. Now assume that the curves $\tilde{u}$ and $\tilde{v}$ are non-intersecting near the punctures with $\gamma \coloneqq \gamma_z=\gamma_w$ (but $m_z$ does not have to be equal to $m_w$). This assumption implies that $m_w\cdot[\tilde{u}; z]$ and $m_z\cdot[\tilde{v}; w]$ are non-intersecting and the difference of their asymptotic representatives $U(s, t)- V\left(s, t + \frac{j}{m}\right)$ is never zero, where $m:=m_zm_w$. Moreover, by Theorem~\ref{asymp_rep1} we can write
$$U(s, t) -V\left(s, t + \frac{j}{m}\right) = e^{\mu_j s} [e'_j (t) + r'_j (s, t)]$$
where $e'_j$ is an eigenvector of the asymptotic operator with eigenvalue $\mu_j < 0$ and where $r'_j$ converges exponentially to zero. In that case, we define \textit{the relative asymptotic winding number} of $\tilde{u}$ and $\tilde{v}$ at $z$ and $w$ respectively relative to the trivialization $\Phi$ as
$$\wind^{\Phi}_{\mathrm{rel}} \left(m_w\cdot[\tilde{u}; z], m_z\cdot[\tilde{v}; w]\right) =
\sum_{j=0}^{m-1} \wind\left(\Phi^{-1} e'_j \right).$$

\begin{remark}
Although we have given an algebraic definition of $\Delta_1$, $\Delta_2$, and $\wind^{\Phi}_{\mathrm{rel}}$, there also exists a geometric interpretation. See \cite[Section~3.2]{Siefinter2009} for a detailed account of this perspective. See also the footnotes on pages 1040 and 1042 of \cite{Siefinter2009}, as well as the paragraph following \cite[Equation~3.29]{Siefinter2009} on page 1052, for brief summaries of these descriptions.
\end{remark}

\begin{lemma}\cite[Section~3.1]{Siefinter2009}\label{relasymp_nonzero}
Let $[\tilde{u}; z]$ and $[\tilde{v}; w]$ be defined as above. Then
$$\wind^{\Phi}_{\mathrm{rel}}\left(m_w\cdot[\tilde{u}; z], m_z\cdot[\tilde{v}; w]\right)=\wind^{\Phi}_{\mathrm{rel}}\left(m_z\cdot[\tilde{v}; w],m_w\cdot[\tilde{u}; z]\right)$$ and the quantity
\begin{equation*}
       m_z m_w \max\left\{\frac{\wind^{\Phi}_{\infty}(u; z)}{m_z}, \frac{\wind^{\Phi}_{\infty}(v; w)}{m_w}\right\} 
       -\frac{1}{m_z m_w} \wind^{\Phi}_{\mathrm{rel}} \left(m_w {\cdot} [u; z], m_z {\cdot}[v; w]\right)
\end{equation*}
is non-negative and independent of the choice of trivialization, and is positive only if $e_1 (m_w {\cdot} \tilde{u}; z)= j *_{m} e_1 (m_z{\cdot}\tilde{v}; w)$ for some $j \in \mathbb{Z}_{m}$, where $m:=m_zm_w$ and $*_{m}$ denotes the natural $\mathbb{Z}_{m}$-action on  $\gamma^{*}\xi^{H}$.
\end{lemma}
We now define the \textit{relative asymptotic self-intersection index}  $\delta_{\infty}([\tilde{u}; z], [\tilde{v}; w])$ as follows,
\begin{equation*}
    \begin{split}
       \delta_{\infty}([\tilde{u}; z], [\tilde{v}; w]) \coloneqq & m_z m_w \max\left\{\frac{\alpha^{\Phi}(\gamma^{m_z})}{m_z}, \frac{\alpha^{\Phi}(\gamma^{m_w})}{m_w}\right\} \\ 
       -&\frac{1}{m_z m_w} \wind^{\Phi}_{\mathrm{rel}} \left(m_w \cdot [u; z], m_z \cdot [v; w]\right) 
    \end{split}
\end{equation*}

\begin{remark}
We assumed above that the puncture $z$ (resp. $w$) of the curves $\tilde{u}$ (resp. $\tilde{v}$) is positive. The definition of asymptotic, secondary and relative winding numbers can be extended when the punctures considered are both negative. These winding numbers satisfy properties similar to the ones given above.
\end{remark}

\subsection{Singularity index and intersection number}\label{intersection number and singularity index}

In this section we recall the definitions of generalised singularity index and intersection number of pseudoholomorphic curves given by R. Siefring in \cite{Siefinter2009}. The total asymptotic self-intersection (or singularity) index $\delta_{\infty}(\tilde{u})$ of a pseudoholomorphic curve $\tilde{u}: \Sigma\setminus \Gamma \rightarrow \widehat{W}$ is defined as follows,

\begin{equation*}
    \begin{split}
\delta_{\infty}(\tilde{u})=\sum_{z\in \Gamma} \delta_{\infty}(\tilde{u}; z) &+ \frac{1}{2}\sum\limits_{\substack{(z,w)\in \Gamma_+ \times \Gamma_+,\\ z \neq w, \gamma_z =\gamma_w}}\delta_{\infty}\left([\tilde{u}; z], [\tilde{u}; w]\right)\\
&+ \frac{1}{2}\sum\limits_{\substack{(z,w)\in \Gamma_- \times \Gamma_-,\\ z \neq w, \gamma_z =\gamma_w}}\delta_{\infty}\left([\tilde{u}; z], [\tilde{u}; w]\right).
    \end{split}
\end{equation*}

Combining this with the usual singularity index $\delta(\tilde{u})$ which is the algebraic count of singular and double points of $\tilde{u}$, we define the generalised singularity index $\mathrm{sing}(\tilde{u})$ of a punctured pseudoholomorphic curve $\tilde{u}$ as
$$\mathrm{sing}(\tilde{u})\coloneqq \delta(\tilde{u}) + \delta_{\infty}(\tilde{u}).$$
The singularity index `$\mathrm{sing}$' is nonnegative-integer-valued, and equals zero for a given curve if and only if that curve is embedded and has total asymptotic self-intersection index equal to zero.\\

Having defined the asymptotic self-intersection index, we now define the intersection number of the pseudoholomorphic curves $\tilde{u}: \Sigma\setminus \Gamma \rightarrow \widehat{W}$ and $\tilde{v}: \Sigma'\setminus \Gamma' \rightarrow \widehat{W}$.

\begin{equation*}
    \begin{split}
    \tilde{u}*\tilde{v}:= \mathrm{int}(\tilde{u}, \tilde{v})  &+ \frac{1}{2}\sum\limits_{\substack{(z,w)\in \Gamma_+ \times \Gamma_+',\\ \gamma_z =\gamma_w}}\delta_{\infty}\left([\tilde{u}; z], [\tilde{v}; w]\right)\\
    &+ \frac{1}{2}\sum\limits_{\substack{(z,w)\in \Gamma_- \times \Gamma_-',\\ \gamma_z =\gamma_w}}\delta_{\infty}\left([\tilde{u}; z], [\tilde{v}; w]\right).
    \end{split}
\end{equation*}
where $\mathrm{int}(\tilde{u}, \tilde{v})$ is the usual intersection number.

\begin{theorem}\cite[Proposition~4.3]{Siefinter2009}
The generalized intersection number $\tilde{u}*\tilde{v}$ defined above is symmetric, bilinear and depends only on the homotopy classes of $\tilde{u}$ and $\tilde{v}$.
\end{theorem}

\begin{remark}\label{rem:GeometricDecrpition}
We could compute this generalized intersection number $\tilde{u}*\tilde{v}$ by perturbing one of the pseudoholomorphic curves near the ends and counting their algebraic intersection number which depends on the direction of the perturbation. We could then add suitable winding numbers to this intersection number to remove the dependence on the direction of perturbation. This description is very geometric but we chose to avoid this as the results later use the explicit description of intersection number in terms of these winding numbers. Refer \cite[Chapter~4]{Wendlcontact3mani2020} or \cite[Section~3.2]{Siefinter2009} for the geometric treatment of this generalised intersection number.
\end{remark}

\section{Stratification of the moduli space} \label{section: Stratification of the moduli space} In this section, we define a sequence of evaluation maps on the moduli space of pseudoholomorphic curves mapping into the eigenspaces of asymptotic operators and prove its regularity. The specific analytical details and notations rely heavily on Wendl's SFT lecture notes \cite{SFTwendl}.\\
\vspace{1pt}

The linearization of the section $\cS$ defined in Section~\ref{Moduli space of pseudoholomorphic curves} at $(j,\tilde{u})\in T \times \cB^{k,p,\delta}$ is

$$D\cS(j,\tilde{u})\colon T_{j} T \oplus W^{k,p,\delta}(\dot{\Sigma},\tilde{u}^*T\widehat{W}) \oplus V_{\Gamma} \rightarrow \cE_{j,\tilde{u}}^{k-1,p,\delta}$$
$$(y,\eta) \mapsto J(\tilde{u})\circ d\tilde{u} \circ y + D_{\tilde{u}} \eta + \cF $$
where $D_{\tilde{u}}$ is the linearized Cauchy-Riemann operator (See~\cite[Section~2.1]{SFTwendl}) and $\cF$ represents the projection of the linearization defined over $V_{\Gamma}$.\\

Given a puncture $z\in \Gamma_+$, we define the `asymptotic' evaluation map on the moduli space $\cM^*(\widehat{W},J)$, that maps  into the eigenspace $E_{\lambda,z}$ of the asymptotic operator $A_{z}\coloneqq A_{\gamma_z^{m_z},J}$, where  $\lambda = \sigma_{\max}^-(\gamma_z^{m_z})$, as follows.
Suppose, $(U, \psi)$ is the asymptotic representative of $\tilde{u} \in \cM^*(\widehat{W},J)$ satisfying
\begin{gather*}
u(\psi(s, t)) = \exp_{\gamma_z^{m_z}(t)} U(s,t),\\
U(s, t) = e^{\mu s} [e(t) + r (s, t)].
\end{gather*}
Then, the evaluation map is defined as 
\begin{equation}\label{evaluation map definfition}
e_{\lambda,z}(\tilde{u})= \begin{cases} 
      e(t) & \mathrm{if\ } \mu=\lambda\\
      0 & \mathrm{if\ } \mu <\lambda
   \end{cases}
\end{equation}
Note that the map $e_{\lambda,z}\colon \cM^*(\widehat{W},J) \rightarrow E_{\lambda,z}$ depends only on $[\tilde{u},z]$. This evaluation map $e_{\lambda,z}$ could also be defined as the limit $\lim\limits_{s\rightarrow \infty}e^{-\lambda s} U(s,t)$ and by Theorem~\ref{asymp_rep1}, this limit takes values in $E_{\lambda,z}$. This definition makes it more evident that the evaluation map is smooth. While the definition of the asymptotic evaluation map and the subsequent proofs in this section are specifically provided for positive punctures, it's important to note that they naturally extend to negative punctures as well.\\

The evaluation map is generally not a submersion. Therefore, we introduce a suitable perturbation of a fixed almost complex structure $J_{\mathrm{fix}}$ to establish transversality. Given an open subset $O\subset W$, the subset 
$$\cJ_O\coloneqq \left\{ J \in \cJ(\widehat{W},\omega_{\phi})\mid J=J_{\mathrm{fix}} \mathrm{\ on \ } \widehat{W}\setminus O\right\}$$ 
containing $J_{\mathrm{fix}} \in \cJ(\widehat{W},\omega_{\phi})$ with $C^{\infty}$-topology is a Fr\'echet manifold. The tangent space at $J_0 \in \cJ_O$ given by
$$T_{J_0} \cJ_O\coloneqq \left\{Y \in \Gamma\left(\End_{\mathbb{C}}\left( T\widehat{W}, J_0\right)\right) \mid  Y |_{\widehat{W}\setminus O}= 0 \mathrm{\ and\ } \omega_{\phi}(\cdot, Y\cdot)+ \omega_{\phi}(Y\cdot,\cdot)=0\right\}.$$

Since Fr\'echet manifolds are not ideal for applying the regular value theorem, we will turn to Floer's $C_{\epsilon}$ space. Given a sequence of positive numbers $\epsilon\coloneqq (\epsilon_{\ell})_{\ell=0}^{\infty}$ with $\epsilon_{\ell} \rightarrow 0$ as $\ell \rightarrow \infty$ and $c>0$, the space of ``$C_{\epsilon}$ small perturbations" of $J_0$ is given by
$$\cJ_O^{\epsilon}\coloneqq \left\{ J_Y \in \cJ_O \mid Y \in T_{J_0}\cJ_O \mathrm{\ with\ } \|Y\|_{C_{\epsilon}}<c \right\},$$
where 
$$\|Y\|_{C_{\epsilon}} = \sum_{\ell=0}^{\infty} \epsilon_{\ell} \|Y\|_{C^{\ell}(O)}$$
and the map $Y\rightarrow J_Y$ is given by 
\[Y\mapsto J_Y\coloneqq \left(1+\frac{1}{2}J_0Y\right)J_0 \left(1+\frac{1}{2}J_0Y\right)^{-1}.\]

This makes $\cJ_O^{\epsilon}$ a separable Banach manifold which can be identified with a subset of the aforementioned Fr\'echet space and we have a continuous inclusion $\cJ^{\epsilon}_O \hookrightarrow \cJ_O$. Based on these definitions, we define the universal moduli space as $$\cM(\widehat{W},\cJ^{\epsilon}_O) =\left\{ (\tilde{u},J) \mid \tilde{u} \in \cM^*(\widehat{W},J) \mathrm{\ and \ } J\in \cJ^{\epsilon}_O \right\}.$$ 

Recall that the definition of the moduli space $\mathcal{M}^*(\widehat{W},J)$ (see Section~\ref{Moduli space of pseudoholomorphic curves}) also requires the choice of an open subset of $W$. For this purpose, we will use the same open set $O \subset W$.

Note that the almost complex structure $J_{0}$ restricted to the open set  ${(0,\infty)\times M^{\pm}}\subset\widehat{W}$ is invariant under the natural translation map. Thus inducing an $\omega^{\pm}$-compatible complex structure on the sub-bundle $\xi^{H^{\pm}}\subset TM^{+}$, denoted also by $J_{0}$. The map $e_{\lambda,z}$ admits a natural extension to the universal moduli space $\cM(\widehat{W},\cJ^{\epsilon}_O)$ as every $J\in \cJ^{\epsilon}_O$ equals $J_{0}$ when restricted to ${(0,\infty)\times M^+}$ and $(-\infty, 0)\times M^-$. It is crucial that all almost complex structures $J$ have the same restriction near the boundary; otherwise, the asymptotic operator, and hence its eigenspaces, would depend on $J$. In that case, it would not be possible to extend the map $e_{\lambda,z}$ over the universal moduli space, as the codomain would vary with $J$. By fixing this choice, we obtain a well-defined extension.\\

We proceed to describe the map 
\[
(e_{\lambda,z})_*: T_{\tilde{u}}\mathcal{M}^*(\widehat{W},J) \to E_{\lambda,z},
\]
together with an extension of its domain that will be required in subsequent arguments. Let 
\[
\zeta \in T_{\tilde{u}}\mathcal{M}^*(\widehat{W},J) \subset W^{k,p,\delta}\!\left(\dot{\Sigma}, \tilde{u}^*T\widehat{W}\right) \oplus V_{\Gamma}.
\]
Throughout this section, we tacitly identify elements of $T_{\tilde{u}}\mathcal{M}^*(\widehat{W},J)$ with their projection onto $W^{k,p,\delta}(\dot{\Sigma}, \tilde{u}^*T\widehat{W})$. The tangent bundle of the cylindrical end admits the canonical splitting
\[
T\big((0,\infty)\times M^+\big)=\varepsilon \oplus \xi^H,
\qquad \varepsilon=\operatorname{span}\{\partial_r, X_H\}.
\]
Consequently, for a cylindrical neighborhood $\mathcal{O}$ of $z$, we may regard
\[
\zeta(s,t) \in W^{k,p,\delta}\big(\mathcal{O}, \tilde{u}^*\varepsilon \oplus \tilde{u}^*\xi^H\big).
\]
By \cite[Theorem~A.1]{Siefinter2008}, the $\xi^H$ component of $\zeta$, denoted by $\zeta(s,t)_{\xi^H}$, admits an asymptotic expansion of the form
\[
\zeta(s,t)_{\xi^H}=e^{\lambda s}\big[e(t)+r(s,t)\big],
\]
where $\lambda$, $e$, and $r$ are as in Theorem~\ref{asymp_rep1}. In particular, the leading-order term determines a well-defined element $e(t)\in E_{\lambda,z}$. Therefore, we have
\[
(e_{\lambda,z})_*(\zeta) := e(t).
\]

We now introduce an extension of the domain of this map. Observe that
\[
W^{k,p,\delta}\big(\mathcal{O}, \tilde{u}^*T((0,\infty)\times M^+)\big)
= W^{k,p,\delta}\big(\mathcal{O}, \tilde{u}^*\varepsilon\big) \oplus W^{k,p,\delta}\big(\mathcal{O}, \tilde{u}^*\xi^H\big).
\]
Define the subspace
\[
\mathcal{W} \coloneqq \left\{ V \in W^{k,p,\delta}(\mathcal{O}, \tilde{u}^*\xi^H) \ \middle|\ \lim_{s\to\infty} e^{-\lambda s} V(s,t)\ \text{exists in } C^\infty(S^1) \right\}.
\]
For $\zeta \in W^{k,p,\delta}\big(\mathcal{O}, \tilde{u}^*\varepsilon\big) \oplus \mathcal{W}$, we define the extended map by
\[
\zeta \longmapsto \operatorname{pr}\!\left( \lim_{s\to\infty} e^{-\lambda s} \zeta_{\xi^H}(s,t) \right),
\]
where $\zeta_{\xi^H}$ denotes the $\xi^H$ component of $\zeta$, and $\operatorname{pr}$ is the projection onto $E_{\lambda,z}$. By construction, this extension agrees with $(e_{\lambda,z})_*$ when restricted to $T_{\tilde{u}}\mathcal{M}^*(\widehat{W},J)$.

\begin{lemma}\label{local_existence}
For any $z\in\Gamma$, $\lambda = \sigma^-_{\max}(\gamma_z^{m_z})$ and $v\in E_{\lambda,z}$, there exist a cylindrical neighbourhood $\cO \subset \Sigma\setminus \Gamma$ of $z$ and a section $\sigma: \cO \rightarrow \tilde{u}^*T\widehat{W}$ satisfying two conditions: $(e_{\lambda,z})_*(\sigma)=v$ and $D_{\tilde{u}}(\sigma)$ has exponential decay strictly faster than $\lambda$.
\end{lemma}

In the above lemma, we choose the cylindrical neighbourhood $\cO$ sufficiently small so that $\tilde{u}(\cO) \subset (0,\infty)\times M^+$, and such that the extension of $(e_{\lambda,z})_*$ defined above is well-defined on $\sigma$.

\begin{proof}
Select a sufficiently small $\delta > 0$, specifically such that $\delta < -\lambda$. Then, we will focus on the linearization's restriction to the subspace $W^{k,p,\delta}\left(\dot{\Sigma},\tilde{u}^*T\widehat{W}\right)$:
$$D_{\tilde{u}}: W^{k,p,\delta}\left(\dot{\Sigma},\tilde{u}^*T\widehat{W}\right) \rightarrow W^{k-1,p,\delta}\left(\dot{\Sigma}, \bigwedge\nolimits^{0,1} T^*\dot{\Sigma}\otimes \tilde{u}^*T\widehat{W}\right).$$
 
Since $\tilde{u}$ is asymptotically cylindrical, there exists a cylindrical neighbourhood $\cO \subset \dot{\Sigma}$ of $z$ such that $\tilde{u}(\cO) \subset (0,\infty)\times M^+ \subset \widehat{W}$. Upon further restriction of the linearization to the sections defined on $\cO$, we obtain the following operator:
$$D': W^{k,p,\delta}\left(\cO,\tilde{u}^*T\left((0,\infty)\times M^+\right)\right) \rightarrow W^{k-1,p,\delta}\left(\cO, \bigwedge\nolimits^{0,1} T^*\dot{\Sigma}\otimes \tilde{u}^*T\left((0,\infty)\times M^+\right)\right).$$

Recall that the tangent space of $(0,\infty)\times M^+$ has a natural splitting into $\varepsilon \oplus \xi^H$, where $\varepsilon$ is the subbundle spanned by $\partial_r$ and the Reeb vector field $X_{H}$. As a result, the operator $D'$ can be viewed as 
$$D': W^{k,p,\delta}\left(\cO,\tilde{u}^*\varepsilon \oplus \tilde{u}^*\xi^H \right) \rightarrow W^{k-1,p,\delta}\left(\cO, \left(\bigwedge\nolimits^{0,1} T^*\dot{\Sigma}\otimes \tilde{u}^*\varepsilon\right) \bigoplus \left(\bigwedge\nolimits^{0,1} T^*\dot{\Sigma}\otimes \tilde{u}^*\xi^H\right)\right).$$

Note that the vector bundles $\tilde{u}^*\varepsilon$ and $\tilde{u}^*\xi^H$(which equals $u^*\xi^H$, as this pullback depends only on $u$) over $\cO$ are unitary trivializable as the cylindrical neighbourhood $\cO$ is homotopic to $S^1$.

By \cite[Proposition~7.9]{SFTwendl} the operator $D'$ is $C^{\infty}$-asymptotic to $(-i\partial_t \oplus A_z)$.  Specifically, if $D'$ is represented by $\overline{\partial}+S(s,t)$ and if $(-i\partial_t \oplus A_z)$ is represented by $-J_0 \partial_t -S_{\infty}(t)$ with $S,S_{\infty} \in C^{\infty}(\cO,\End_{\mathrm{sym}}(\mathbb{R}^4))$, then
$$\lim\limits_{s\rightarrow \infty}S(s,t) \rightarrow S_{\infty}(t)$$
exponentially in $C^{\infty}(S^1)$. The section $(0,\eta)\in W^{k,p,\delta}\left(\cO,\tilde{u}^*\varepsilon \oplus \tilde{u}^*\xi^H \right)$, where $\eta=e^{\lambda s}(v(t))$, satisfies $\overline{\partial}\eta+S_{\infty}(t)(0,\eta)=0$. This section will be represented simply by $\eta$ when there is no ambiguity. Thus, \begin{equation*}
    \begin{split}
        \overline{\partial}\eta+S(s,t)(\eta) &= \overline{\partial}\eta+S_{\infty}(t)(\eta)+ (S(s,t)-S_{\infty}(t))(\eta)\\
        &= \left(S(s,t)-S_{\infty}(t)\right)(\eta).
    \end{split}
\end{equation*} 

Since $\left(S(s,t)-S_{\infty}(t)\right)$ decays exponentially, the section $\left(S(s,t)-S_{\infty}(t)\right)(\eta)$ decays exponentially strictly faster than $\lambda$. Thus, the section $(0,\eta)$ is the required section $\sigma$ appearing in the statement of the theorem.
\end{proof}

\begin{proposition}   
\label{regularity of evaluation map}
For any $z\in \Gamma_+$ and $\lambda = \sigma_{\max}^-(\gamma_z^{m_z})$, the evaluation map $e_{\lambda,z}\colon \cM(\widehat{W},\cJ^{\epsilon}_O) \rightarrow E_{\lambda,z}$ is submersion.
\end{proposition}

\begin{proof}

We begin by examining the image of the operator
$$(\xi,Y)\xrightarrow[]{L}D_{\tilde{u}}\xi+\frac{1}{2}Y(\tilde{u})\circ d\tilde{u}\circ j_{\Sigma},$$
restricted to the set of all pairs $(\xi,Y) \in W^{k,p,\delta}\left(\dot{\Sigma}, \tilde{u}^*T\widehat{W}\right)\times  T_J(J^{\epsilon}_O)$, where $\supp{Y}\subset O$. Thus we define the subspace
$$Z^k \coloneqq \left\{ D_{\tilde{u}}\xi+\frac{1}{2}Y(\tilde{u})\circ d\tilde{u}\circ j_{\Sigma} \mid \xi \in W^{k,p,\delta}\left(\dot{\Sigma}, \tilde{u}^*T\widehat{W}\right),\mathrm{\ and \ }\supp Y \subset O \right\}.$$

By Lemma~\ref{local_existence}, we choose a cylindrical neighbourhood $\cO$ around $z$ and  $\xi: \cO \rightarrow \tilde{u}^*T\widehat{W}$ such that  $(e_{\lambda,z})_*(\xi)=v$ and $D_{\tilde{u}}(\xi)$ decays exponentially faster than $\lambda$. We now extend this to the entire Riemann surface using a bump function $\beta:\dot{\Sigma} \rightarrow [0,1]$ such that $\beta=1$ on another cylindrical neighbourhood $W\subset \cO$ and $\supp (\beta) \subset \cO$.

Let $c<0$ denote the exponential rate of decay of $D_{\tilde{u}}(\xi)$, and let $\delta>0$ be such that $-\delta \not\in \sigma(A_z)$ and $c<-\delta<\lambda$. Let us suppose for a moment that we have shown that $Z^k = W^{k-1,p,\delta}\left(\dot{\Sigma}, \bigwedge\nolimits^{0,1}\tilde{u}^*T\widehat{W}\right)$, we will show hwo this implies that 
$e_{\lambda,z}\colon \cM(\widehat{W},\cJ^{\epsilon}_O) \rightarrow E_{\lambda,z}$ is submersion. 
The statement $Z^k = W^{k-1,p,\delta}\left(\dot{\Sigma}, \bigwedge\nolimits^{0,1}\tilde{u}^*T\widehat{W}\right)$ implies that $D_{\tilde{u}}(\beta \xi) \in \im(L)$ i.e, there exists $(\eta,Y) \in W^{k,p,\delta}\left(\dot{\Sigma}, \tilde{u}^*T\widehat{W}\right)\times  T_J(J^{\epsilon}_O)$ such that,

$$D_{\tilde{u}}\eta+\frac{1}{2}Y(\tilde{u})\circ d\tilde{u}\circ j_{\Sigma}=D_{\tilde{u}}(\beta \xi)$$
and $\eta-\beta \xi$ is the required solution in $W^{k,p,\delta}\left(\dot{\Sigma}, \tilde{u}^*T\widehat{W}\right)$ with $(e_{\lambda,z})_*(\eta-\beta\xi)=v$. The proof establishing $Z^k = W^{k-1,p,\delta}\left(\dot{\Sigma}, \bigwedge\nolimits^{0,1}\tilde{u}^*T\widehat{W}\right)$ is exactly the proof given in \cite[Section~8.3]{SFTwendl}. Assume that $k = 1$. Since the operator
$$D_{\tilde{u}} \colon W^{1,p,\delta}\left(\dot{\Sigma},\tilde{u}^*T\widehat{W}\right) \rightarrow L^{p,\delta}\left(\dot{\Sigma}, \bigwedge\nolimits^{0,1}\tilde{u}^*T\widehat{W}\right)$$
is a Fredholm for any $\delta>0$ with $-\delta \not\in \sigma(A_z)$, the image of $D_{\tilde{u}}$ is closed and has finite codimension. Hence, the subspace $Z^1$ is closed in $L^{p,\delta} \left(\dot{\Sigma}, \bigwedge\nolimits^{0,1} \tilde{u}^*T\widehat{W}\right)$. By the Hahn--Banach theorem, there exists $\theta \in L^{q,-\delta} \left(\dot{\Sigma}, \bigwedge\nolimits^{0,1} \tilde{u}^*T\widehat{W}\right)$ that annihilates $Z^1$, where $\frac{1}{p} + \frac{1}{q} = 1$. This implies 

$$\left(D_{\tilde{u}}(\eta), \theta\right)_{L^{2}} = 0 \mathrm{\ for\ all\ } \eta \in W^{1,p,\delta} \left(\dot{\Sigma},\tilde{u} ^* T \widehat{W})\right),$$
and 
$$\left(Y(\tilde{u}) \circ d\tilde{u} \circ j_{\Sigma} , \theta \right)_{L^2} = 0 \mathrm{\ for\ all\ } Y \in T_J J^{\epsilon}_O.$$

The first relation implies that $\theta$ is a weak solution to the formal adjoint equation $D^*_{\tilde{u}} \theta = 0$. By elliptic regularity and the similarity principle, we conclude that $\theta$ is smooth function with isolated zeroes. Next, we will demonstrate that the existence of an injective point $z_0 \in \dot{\Sigma}$ satisfying $\tilde{u} (z_0) \in O$ contradicts the second relation.\\

We assume without loss of generality that $\theta(z_0)\neq 0$ for an injective point $z_0 \in \dot{\Sigma}$ satisfying $\tilde{u} (z_0) \in O$. Since $d\tilde{u} (z_0) \neq 0$, employing a standard lemma from symplectic linear algebra \cite[Lemma~3.2.2]{mcduff2012j} helps to find a smooth section $Y \in T_{J} J^{\epsilon}_O$. This section’s value at $\tilde{u}(z_0)$ is specifically chosen so that $Y(\tilde{u}) \circ d\tilde{u} \circ j= \theta$ at $z_0$.
This choice implies that the pointwise inner product $\left(Y(\tilde{u}) \circ d\tilde{u} \circ j, \theta\right)$ is positive in some neighborhood of $z_0$. One can then multiply $Y$ by a bump function to produce a section (still denoted by $Y$) of class $C_{\epsilon}$ and that the pointwise inner product of $\left(Y(\tilde{u}) \circ d\tilde{u} \circ j, \theta \right)$ is positive near $z_0$ but vanishes everywhere else. Indeed, this construction hinges upon the assumption that $z_0$ is an injective point. This contradicts the second condition in and thus completes the proof for $k= 1$.\\

Lets proceed to the case $k>1$. Suppose $\alpha \in W^{k-1,p,\delta}\left(\dot{\Sigma}, \bigwedge\nolimits^{0,1}\tilde{u}^*T\widehat{W}\right)$. In the case of $k=1$, surjectivity implies the existence of $\eta \in W^{1,p,\delta}$ and
$Y \in T_{J} J^{\epsilon}_U$ with $D_{\tilde{u}} \eta + \frac{1}{2}Y(\tilde{u}) \circ d\tilde{u} \circ j = \alpha$. Then, by elliptic regularity $\eta\in W^{k,p,\delta}$. This proves the surjectivity for arbitrary $k\in \mathbb{N}$ and $p\in(1,\infty)$.

\end{proof}

Now, applying the usual Sard-Smale argument we can show there exist a comeagre set $\cJ^1_{\mathrm{reg},z} \subset \cJ^{\epsilon}_{O}$ such that $e_{\lambda,z}^{-1}(0) \subset \cM^*(\widehat{W},J)$ is a submanifold for every $J \in \cJ^1_{\mathrm{reg},z}$ and denote this submanifold by $\cM_{\lambda'}^z(\widehat{W},J)$ where $\lambda'$ is the eigenvalue immediately below $\lambda$. Note that the asymptotic representative $(U,\psi)$ the pseudoholomorphic curves in $\cM_{\lambda'}^z(\widehat{W},J)$ takes the form 
$$U(s, t) = e^{\mu s} [e(t) + r (s, t)]$$
where $\mu \leq \lambda'$ and therefore the notation $\cM_{\lambda'}^z(\widehat{W},J)$ is justified. The elements within the tangent space of $\cM_{\lambda'}^z(\widehat{W},J)$ also fulfill a similar asymptotic behaviour.\\

On the submanifold $\cM^z_{\lambda'}(\widehat{W},\cJ^{\epsilon}_O)\coloneqq e_{\lambda,z}^{-1}(0)\subset \cM^*(\widehat{W},\cJ^{\epsilon}_O)$ of the universal moduli space, we define an another evaluation map $e_{\lambda',z}$ into the eigenspace $E_{\lambda',z}$. Let $(U, \psi)$ be the asymptotic representative of the pair $(\tilde{u},J) \in \cM^z_{\lambda'}(\widehat{W},\cJ^{\epsilon}_O)$ satisfying
$$U(s, t) = e^{\mu s} [e(t) + r (s, t)]$$
where $\mu \leq \lambda'< \lambda$. We define the evaluation map $e_{\lambda',z}$ as follows,

\[ e_{\lambda',z}(u,J)= \begin{cases} 
      e(t) & \mathrm{if\ } \mu=\lambda'\\
      0 & \mathrm{if\ } \mu <\lambda'
   \end{cases}.
\]

The proofs of Lemma~\ref{local_existence} and Theorem~\ref{regularity of evaluation map} can be easily adapted with very minimal changes to demonstrate the submersion of the evaluation map $e_{\lambda',z}$. By applying the Sard--Smale theorem again, we obtain a comeagre set $\cJ^2_{\mathrm{reg},z} \subset \cJ^1_{\mathrm{reg},z} \subset \cJ^{\epsilon}_{O}$ such that for every $J\in \cJ^2_{\mathrm{reg},z}$, we have submanifolds
$$\cM^*(\widehat{W},J) = \cM^z_{\lambda}(\widehat{W},J) \supset \cM^z_{\lambda'}(\widehat{W},J)\supset \cM^z_{\lambda{''}}(\widehat{W},J).$$

In this way, we can recursively define a sequence of maps $e_{\lambda_i,z}:\cM^z_{\lambda_{i}}(\widehat{W},\cJ^{\epsilon}_O) \rightarrow E_{\lambda_{i},z}$ for every negative eigenvalue $0>\lambda_1>\lambda_2 > \dots$ of the asymptotic operator $A_z$. Furthermore, we obtain co-meagre sets $\cJ^{\epsilon}_{O} \supset \cJ^1_{\mathrm{reg},z} \supset \cJ^2_{\mathrm{reg},z} \dots$ such that for every $ J$ in the comeagre subset $\bigcap\limits_i \cJ^i_{\mathrm{reg},z}$, we have a filtration of the moduli space $\cM^*(\widehat{W},J)$,
$$\cM^*(\widehat{W},J) = \cM^z_{\lambda_1}(\widehat{W},J) \supset \cM^z_{\lambda_2}(\widehat{W},J) \supset \cM^z_{\lambda_3}(\widehat{W},J) \dots $$

We will now apply the Taubes's trick\cite[Section~7.7]{SFTwendl} to this result on $C_{\epsilon}$-space to obtain the following theorem,

\begin{theorem} \label{stratification_symplectization}
For any $z\in \Gamma_+$, there exist a comeagre subset $\cJ_{\mathrm{reg},z} \subset \cJ_O$ such that for every $J \in \cJ_{\mathrm{reg},z}$, the moduli space $\cM^*(\widehat{W},J)$ admits a filtration 
$$\cM^*(\widehat{W},J) = \cM^z_{\lambda_1}(\widehat{W},J) \supset \cM^z_{\lambda_2}(\widehat{W},J) \supset \cM^z_{\lambda_3}(\widehat{W},J) \dots $$
where $0>\lambda_1> \lambda_2> \dots $ are the negative eigenvalues of $A_z$. A pseudoholomorphic curve $\tilde{u}\in \cM^*(\widehat{W},J)$ belongs to $\cM^z_{\lambda_i}(\widehat{W},J)$ if and only if its asymptotic representative $(U, \psi)$ takes the form
$$U(s, t) = e^{\mu s} [e(t) + r (s, t)]$$
where $\mu \leq \lambda_i$.
\end{theorem}

Our next objective is to provide a filtration for the moduli space \(\mathcal{M}^*(M,J)\). Given an open subset $O\subset M$, a suitable class perturbation of a fixed almost complex structure $J_{\mathrm{fix}}$ is defined by
$$\cJ_O\coloneqq \left\{ J \in \cJ(M,H) \mid J=J_{\mathrm{fix}} \mathrm{\ on \ } \R \times (M\setminus O)\right\}.$$ 
However, there's a caveat: for the evaluation map to be well-defined on the universal moduli space, the complement of the open subset $U$ must contain the Reeb orbits involved in defining the moduli space. Once this condition is met, we achieve a well-defined evaluation map on the universal moduli space as the asymptotic operators remain unchanged under the $C_{\epsilon}$ small perturbation of almost complex structures.\\

With this resolved, the subsequent steps of the regularity proof follow a strategy very similar to that employed in Theorem~\ref{regularity of evaluation map}. The only difference lies in the fact that demonstrating that $Z^k=W^{k-1,p,\delta}\left(\dot{\Sigma}, \bigwedge\nolimits^{0,1}\tilde{u}^*T(\R \times M)\right)$ is little bit more complicated. But this has been described in great detail in \cite[Section~8.3]{SFTwendl}. Having established this, we arrive at the following theorem.

\begin{theorem} \label{stratification}
For any open subset $O\subset M$ whose complement contains the Reeb orbits and any $z\in \Gamma_+$, there exist a comeagre subset $\cJ_{\mathrm{reg},z} \subset \cJ_O$ such that for every $J \in \cJ_{\mathrm{reg},z}$, the moduli space $\cM^*(M,J)$ admits a filtration
$$\cM^*(M,J) = \cM^z_{\lambda_1}(M,J) \supset \cM^z_{\lambda_2}(M,J) \supset \cM^z_{\lambda_3}(M,J) \dots $$
where $0>\lambda_1> \lambda_2> \dots $ are the negative eigenvalues of $A_z$. A pseudoholomorphic curve $\tilde{u}\in \cM^*(M,J)$ belongs to $\cM^z_{\lambda_i}(M,J)$ if and only if its asymptotic representative $(U, \psi)$ takes the form
$$U(s, t) = e^{\mu s} [e(t) + r (s, t)]$$
where $\mu \leq \lambda_i$.
\end{theorem}

The filtration provided by Theorems~\ref{stratification_symplectization} and \ref{stratification} has a positive codimension of either 1 or 2 at each step, depending on the dimension of the eigenspaces of the asymptotic operators. These filtrations naturally induces a stratification of the moduli space, where the strata are defined by successive complements $\cM^z_{\lambda_k}(M,J) \setminus \cM^z_{\lambda_{k+1}}(M,J)$.

\section{Subsets of positive asymptotic intersection index} \label{section: Subsets of positive asymptotic intersection index}

In this section, we would like to characterise the subsets where the asymptotic self-intersection indices vanishes and the subsets where the asymptotic contribution to the intersection number vanishes using the stratification given above.

\subsection{Positivity of asymptotic self intersection index}
Recall that the self-intersection index is $\delta_{\infty}(\tilde{u},z) =\frac{1}{2}(\Delta_1+\Delta_2)$ and hence it will be sufficient if we could specify precisely the subset where $\Delta_1$ and $\Delta_2$ vanish. We assume that $m_z>1$, otherwise $\Delta_1$ and $\Delta_2$ identically vanish. The stratification constructed above along with Lemma~\ref{self-intersction=0} helps in describing these subsets.\\

The theme of this subsection would be to present a detailed description of the subsets where $\Delta_1$ and $\Delta_2$ are positive followed by a simplified theorem suitable for applications.

\begin{lemma}
Given $z\in \Gamma_+$, $J \in \cJ_{\mathrm{reg},z}$ and a filtration as in the Theorem~\ref{stratification}
$$\cM^*(\widehat{W},J) = \cM^z_{\lambda_1}(\widehat{W},J) \supset \cM^z_{\lambda_2}(\widehat{W},J) \supset \cM^z_{\lambda_3}(\widehat{W},J) \dots $$
where $0>\lambda_1> \lambda_2> \dots $ are the negative eigenvalues of $A_z$, we have the following:

\begin{enumerate}[topsep=5pt,itemsep=5pt]
    \item[Case 1:] $p(\gamma_z^{m_z})=0$ and $\overline{\sigma}(\gamma_z) \neq m_z$, then $\dim(E_{\lambda_1})=1$ and the submanifold $\cM^z_{\lambda_2}(\widehat{W},J)$ is precisely the subset with $\Delta_1>0$.
    \item[Case 2:] $p(\gamma_z^{m_z})=0$ and $\overline{\sigma}(\gamma_z) = m_z$, then $\dim(E_{\lambda_1})=1$. 
    \begin{itemize}[topsep=5pt,itemsep=5pt]
     \item if $\dim(E_{\lambda_2})=\dim(E_{\lambda_3})=1$, then the submanifold $\cM^z_{\lambda_4}(\widehat{W},J)$ is precisely the subset with $\Delta_1>0$.
    \item if $\dim(E_{\lambda_2})=2$, then the submanifold $\cM^z_{\lambda_3}(\widehat{W},J)$ is precisely the subset with $\Delta_1>0$.
    \end{itemize}
    \item[Case 3:] $p(\gamma_z^{m_z})=1$ and $\overline{\sigma}(\gamma_z) \neq m_z$
    \begin{itemize}[topsep=5pt,itemsep=5pt]
    \item If $\dim(E_{\lambda_1})=2$, then the submanifold $\cM^z_{\lambda_2}(\widehat{W},J)$ is precisely the subset with $\Delta_1>0$.
    \item If $\dim(E_{\lambda_1})=\dim(E_{\lambda_2})=1$, then the submanifold $\cM^z_{\lambda_3}(\widehat{W},J)$ is precisely the subset with $\Delta_1>0$.
    \end{itemize}
    \item[Case 4:] $p(\gamma_z^{m_z})=1$ and $\overline{\sigma}(\gamma_z) = m_z$
    \begin{itemize}[topsep=5pt,itemsep=5pt]
    \item If $\dim(E_{\lambda_1})=2$;
    \begin{itemize}[topsep=5pt,itemsep=5pt]
        \item if $\dim(E_{\lambda_2})=\dim(E_{\lambda_3})=1$, then the submanifold $\cM^z_{\lambda_4}(\widehat{W},J)$ is precisely the subset with $\Delta_1>0$.
        \item if $\dim(E_{\lambda_2})=2$, then the submanifold $\cM^z_{\lambda_3}(\widehat{W},J)$ is precisely the subset with $\Delta_1>0$.
    \end{itemize}
    \item If $\dim(E_{\lambda_1})=\dim(E_{\lambda_2})=1$;
    \begin{itemize}[topsep=5pt,itemsep=5pt]
        \item if $\dim(E_{\lambda_3})=\dim(E_{\lambda_4})=1$, then the submanifold $\cM^z_{\lambda_5}(W,J)$ is precisely the subset with $\Delta_1>0$.
        \item if $\dim(E_{\lambda_3})=2$, then the submanifold $\cM^z_{\lambda_4}(\widehat{W},J)$ is precisely the subset with $\Delta_1>0$.
    \end{itemize}
\end{itemize}

\end{enumerate}

\end{lemma}

The proof of the above lemma follows easily from Theorem~\ref{stratification} and Lemma~\ref{asymptotic_dimension}. The following theorem is a simplified and easy-to-apply version of the above lemma.

\begin{theorem}\label{delta1=0}
For every $z\in \Gamma_+$, $J \in \cJ_{\mathrm{reg},z}$ and a filtration as in the Theorem~\ref{stratification}, the submanifold $\cM^*(\widehat{W},J)\setminus \cM^z_{\lambda_2}(\widehat{W},J)$ is an open dense subset of the moduli space $\cM^*(\widehat{W},J)$ such that $\Delta_1=0$ for every pseudoholomorphic curve $\tilde{u}$ in $\cM^*(\widehat{W},J)\setminus \cM^z_{\lambda_2}(\widehat{W},J)$.
\end{theorem}

This completes the description of subsets with $\Delta_1=0$ and we will now begin to describe the subsets with $\Delta_2=0$. Given a trivialization $\Phi$ of $\gamma_z^* \xi^{H}$, $\Delta_2 (\tilde{u}; z) = 0$  if and only if 
\begin{itemize}[topsep=5pt,itemsep=5pt]
    \item if $\cov(e_1)=\gcd(\wind^{\Phi}(e_1),m_z)=1$ or
    \item if $\wind(e_1)-\wind(e_2)=1$; this implies\\ $\gcd(\cov(e_1), \cov(e_2))=\gcd(\wind^{\Phi}(e_1),\wind^{\Phi}(e_2),m_z)=1$.
\end{itemize}

Though a unitary trivialization is needed to give the conditions, the conditions themselves are independent of the trivialization. This statement is a reformulation of Lemma~\ref{self-intersction=0} and can be easily derived from the proof provided in \cite[Lemma~3.14]{Siefinter2009}. With the above filtration of the moduli space in mind, we will give the following definitions which will be useful in characterizing the subsets with $\Delta_2>0$.

\begin{definition}
Given $\lambda \in \sigma(A_{z})$ with $\lambda<0$ and $\dim(E_\lambda)=2$, we define $G_{\lambda}$ as 
$$G^z_{\lambda}:=\cM^z_{\lambda}(\widehat{W},J)\setminus \cM^z_{\mu}(\widehat{W},J),$$
where $\mu \in \sigma(A_z)$ is the eigenvalue immediately preceding $\lambda$. If $p(\gamma_z^{m_z})=0$, and $\lambda = \sigma_{\max}^-(\gamma_z^{m_z})$ then we have $\dim(E_\lambda)=1$. We then define 
$$G^z_{\lambda}:=\cM^z_{\lambda}(\widehat{W},J)\setminus \cM^z_{\mu}(\widehat{W},J),$$
where $\mu$ is defined as before.
Given a trivialization $\Phi$ of $\gamma_z^* \xi^{H}$ and $\lambda_1, \lambda_2 \in \sigma(A_z)$ with $0>\lambda_1 > \lambda_2$, $\dim{E_{\lambda_1}}=\dim{E_{\lambda_2}}=1$ and $\wind^{\Phi}(E_{\lambda_1})=\wind^{\Phi}(E_{\lambda_2})$. We define $G^z_{\lambda_1}=G^z_{\lambda_2}$ as
$$G^z_{\lambda_1}=G^z_{\lambda_2}:=\cM^z_{\lambda_1}(\widehat{W},J)\setminus \cM_{\mu}(\widehat{W},J),
$$ where $\mu \in \sigma(A_z)$ is the eigenvalue immediately preceding $\lambda_2$.
\end{definition}

By the reformulation of Lemma~\ref{self-intersction=0} given above, every curve $u \in G^z_{\lambda}$ for $\lambda \in \sigma(A_z)$ with $\cov(E_{\lambda})=1$ has $\Delta_2=0$. Thus every curve in the closed subset $\Omega_1$ defined below has $\Delta_2=0$,

$$\Omega_1:=\bigcup_{\substack{\lambda \in \sigma(A_z)\\ \cov(E_{\lambda})=1}} G^z_{\lambda}.$$\\
\vspace{1pt}

We would like to characterize the curves whose asymptotic representative near $z$, as described in Equation~\ref{asymp_rep2}, has two terms and the winding of the eigenvectors appearing in this formula differ precisely by 1. This requires considering following four cases;

\begin{enumerate}[topsep=5pt,itemsep=5pt]
    \item Let $\lambda \in \sigma(A_z)$ with $\cov(E_\lambda)\neq 1$. If $\lambda=\sigma_{\max}^{-}(\gamma_z^{m_z})$ and $p(\gamma_z^{m_z})=0$ or $\dim(E_{\lambda})=2$, with $\dim(E_{\mu})=2$, where $\mu \in \sigma(A_z)$ is the immediate eigenvalue preceding $\lambda$, then we can define the evaluation map 
    $$f_\mu :\cN^z_{\lambda}:=\cM^z_{\lambda}(\widehat{W},J) \setminus \cM^z_{\mu}(\widehat{W},J) \rightarrow E_{\mu},$$ given by \cref{asymp_rep2}.  It follows that $\cP^z_{\lambda}:=\cN^z_{\lambda} \setminus f_{\mu}^{-1}(0)$ has $\Delta_2=0$.
    
    \item Let $\lambda \in \sigma(A_z)$ with $\cov(E_\lambda)\neq 1$. If $\lambda=\sigma_{\max}^{-}(\gamma_z^{m_z})$ and $p(\gamma_z^{m_z})=0$ or $\dim(E_{\lambda})=2$, with $\dim(E_{\mu_1})=\dim(E_{\mu_2})=1$, where $\mu_1,\mu_2 \in \sigma(A_z)$, $\lambda>\mu_1>\mu_2$ with no eigenvalues between them.  Now, we can define the evaluation maps 
    \begin{equation*}
        \begin{split}
            f_{\mu_1} &:\cN^z_{\lambda}:=\cM^z_{\lambda}(\widehat{W},J) \setminus \cM^z_{\mu_1}(\widehat{W},J) \rightarrow E_{\mu_1},\\
            f_{\mu_2} &:\cN^z_{\lambda}:=\cM^z_{\lambda}(\widehat{W},J) \setminus \cM^z_{\mu_1}(\widehat{W},J) \rightarrow E_{\mu_2},
        \end{split}
    \end{equation*}
     given by \cref{asymp_rep2}.  It follows easily that $\cP^z_{\lambda}:=\cN^z_{\lambda} \setminus \left\{f_{\mu_1}^{-1}(0)\cup f_{\mu_2}^{-1}(0)\right\}$ has $\Delta_2=0$.
    
    \item Let $\lambda_1,\lambda_2 \in \sigma(A_z)$ with $0>\lambda_1>\lambda_2$, $\dim(E_{\lambda_1})=\dim(E_{\lambda_2})=1$, $\wind^{\Phi}(E_{\lambda_1})=\wind^{\Phi}(E_{\lambda_2})$, $\cov(E_{\lambda_1})=\cov(E_{\lambda_2}) \neq 1$  and $\dim(E_{\mu})=2$, where $\mu \in \sigma(A_z)$, $\lambda_1>\lambda_2>\mu$ with no eigenvalues in between. 
    Now, we can define the evaluation maps 

    \begin{equation*}
        \begin{split}
            f_\mu &:\cN^z_{\lambda_1}:=\cM^z_{\lambda_1}(\widehat{W},J) \setminus \cM^z_{\lambda_2}(\widehat{W},J) \rightarrow E_{\mu},\\
            g_\mu &:\cN^z_{\lambda_2}:=\cM^z_{\lambda_2}(\widehat{W},J) \setminus \cM^z_{\mu}(\widehat{W},J) \rightarrow E_{\mu},
        \end{split}
    \end{equation*}
     given by \cref{asymp_rep2}.  It follows easily that $\cP^z_{\lambda}:=\left(\cN^z_{\lambda_1} \setminus g_{\mu}^{-1}(0)\right)\sqcup \left(\cN^z_{\lambda_2} \setminus g_{\mu}^{-1}(0)\right)$ has $\Delta_2=0$.
    
    \item Let $\lambda_1, \lambda_2 \in \sigma(A_z)$ with $0>\lambda_1>\lambda_2$, $\dim(E_{\lambda_1})=\dim(E_{\lambda_2})=1$, $\wind^{\Phi}(E_{\lambda_1})=\wind^{\Phi}(E_{\lambda_2})$, $\cov(E_{\lambda_1})=\cov(E_{\lambda_2}) \neq 1$ and $\dim(E_{\mu_1})=\dim(E_{\mu_2})=1$, where $\mu_1,\mu_2 \in \sigma(A_z)$, $\lambda_1> \lambda_2>\mu_1>\mu_2$ with no eigenvalues between them.  Now, we can define the evaluation maps 
    \begin{equation*}
        \begin{split}
            f_{\mu_1} &:\cN^z_{\lambda_1}:=\cM^z_{\lambda_1}(\widehat{W},J) \setminus \cM^z_{\lambda_2}(\widehat{W},J) \rightarrow E_{\mu_1},\\ 
            g_{\mu_1} &:\cN^z_{\lambda_2}:=\cM^z_{\lambda_2}(\widehat{W},J) \setminus \cM^z_{\mu_1}(\widehat{W},J) \rightarrow E_{\mu_1},\\
            f_{\mu_2} &:\cN^z_{\lambda_1}:=\cM^z_{\lambda_1}(\widehat{W},J) \setminus \cM^z_{\lambda_2}(\widehat{W},J) \rightarrow E_{\mu_2},\\
            g_{\mu_2} &:\cN^z_{\lambda_2}:=\cM^z_{\lambda_2}(\widehat{W},J) \setminus \cM^z_{\mu_1}(\widehat{W},J) \rightarrow E_{\mu_2},
        \end{split}
    \end{equation*}

     given by \cref{asymp_rep2}.  It follows easily that $\cP^z_{\lambda}$ defined as $$\cP^z_{\lambda}:=\left(\cN^z_{\lambda_1} \setminus \{f_{\mu_1}^{-1}(0)\cup f_{\mu_2}^{-1}(0)\}\right) \sqcup \left(\cN^z_{\lambda_2} \setminus \{g_{\mu_1}^{-1}(0)\cup g_{\mu_2}^{-1}(0)\}\right)$$ has $\Delta_2=0$.
\end{enumerate}

The evaluation maps $f_{\mu}$'s and $g_{\mu}$'s defined above using \cref{asymp_rep2} are given by explicit formulae for lower order exponents in the expansion of asymptotic representative in \cite[Lemma~3.1]{Siefinter2008}. This shows that $f_{\mu}$'s and $g_{\mu}$'s are well defined and smooth. The above mentioned cases exhaust the possibilities when $\cov(E_{\lambda})\neq 1$ and every curve in the closed subset $\Omega_2$
$$\Omega_2:=\bigcup_{\substack{\lambda \in \sigma(A_z)\\ \cov(E_{\lambda})\neq1}} \cP^z_{\lambda},$$
has $\Delta_2=0$. Thus any curve in $\Omega= \Omega_1 \sqcup \Omega_2$ has $\Delta_2=0$. 
We now present a theorem similar in spirit to Theorem~\ref{delta1=0} but a little more complicated to prove. However, the good thing is that the analysis needed for the proof of the theorem  has already been employed in Theorem~\ref{regularity of evaluation map}.

\begin{theorem} \label{delta2=0}
For every $z\in \Gamma_+$, there exist a comeagre subset of $\cJ_O$ such that for every $J$ in that comeagre subset, the subset of pseudoholomorphic curves in $\cM^*(\widehat{W},J)$ with $\Delta_2=0$ contains an open dense subset.
\end{theorem}

\begin{proof}

The proof can be divided into two parts. Let's start with a simpler case where $\overline{\sigma}(\gamma_z^{m_z})=\cov(\sigma_{\max}^{-}(\gamma_z^{m_z}))=1$.
In this case, the submanifold $\cM^z_{\lambda_1}(\widehat{W},J)\setminus \cM^z_{\lambda_2}(\widehat{W},J)$ is an open dense subset of the moduli space $\cM^*(\widehat{W},J)$ and $\Delta_2=0$ for every pseudoholomorphic curve $\tilde{u}$ in $\cM^z_{\lambda_1}(\widehat{W},J)\setminus \cM^z_{\lambda_2}(\widehat{W},J)$.\\

Now, let's consider the case where $\cov(\sigma_{\max}^{-}(\gamma_z^{m_z}))\neq 1$. In this case, the submanifold $\cM^z_{\lambda_1}(\widehat{W},J)\setminus \cM^z_{\lambda_2}(\widehat{W},J)$ still has curves with $\Delta_2\neq0$. But, the pseudoholomorphic curves in the open subset $P^z_{\lambda_1}$ as described above have $\Delta_2=0$. Thus it would be sufficient to prove that the maps $f_{\mu}$'s and $g_{\mu}$'s are regular as this would mean the submanifolds $f_{\mu}^{-1}(0)$ and $g_{\mu}^{-1}(0)$ have positive codimension. The regularity is proved exactly the same way as Theorem~\ref{regularity of evaluation map}.

\end{proof}

\subsection{Positivity of relative asymptotic self intersection index}
Similar to the previous section, we would like to characterize the subset where the relative asymptotic self intersection index vanishes. We have punctures $z,w \in \Gamma$ at which the pseudoholomorphic curves in $\cM^*(\widehat{W},J)$ are positively asymptotic to the same Reeb orbit $\gamma \coloneqq \gamma_z=\gamma_w$. Assuming that $J\in \cJ_{\mathrm{reg},z}\bigcap \cJ_{\mathrm{reg},w}$ as described in Theorem~\ref{stratification}, we obtain the following filtrations

$$\cM^*(\widehat{W},J) = \cM^z_{\lambda_1}(\widehat{W},J) \supset \cM^z_{\lambda_2}(\widehat{W},J) \supset \cM^z_{\lambda_3}(\widehat{W},J) \dots $$
$$\cM^*(\widehat{W},J) = \cM^w_{\mu_1}(\widehat{W},J) \supset \cM^z_{\mu_2}(\widehat{W},J) \supset \cM^z_{\mu_3}(\widehat{W},J) \dots $$
of the moduli space $\cM^*(\widehat{W},J)$ for $0>\lambda_1> \lambda_2> \dots $(resp. $0>\mu_1> \mu_2> \dots $) the negative eigenvalues of $A_z$(resp. $A_w$).\\

We will provide a description of subsets with vanishing relative self-intersection index in two cases: 
\begin{itemize}
    \item[Case 1:]  Let $m_w\cdot \sigma^-_{\max}(\gamma_z^{m_z}) \neq m_z \cdot \sigma^-_{\max}(\gamma_w^{m_w})$. In this case, it follows directly from Lemma~\ref{self-intersction=0} that every pseudoholomorphic curve $\tilde{u}$ in
    $$\cM^*(\widehat{W},J)\setminus \left(\cM^z_{\lambda_2}(\widehat{W},J) \cup \cM^w_{\mu_2}(\widehat{W},J)\right)$$
    has $\delta_{\infty}([\tilde{u}; z],[\tilde{u};w])=0$.
    \item[Case 2:] Let $m_w\cdot \sigma^-_{\max}(\gamma_z^{m_z}) = m_z \cdot \sigma^-_{\max}(\gamma_w^{m_w}) \eqqcolon \nu$. Here, we define a map $h_i$ for each $i$ satisfying $1\leq i < m_z\cdot m_w$, 
    $$h_i: \cM^*(\widehat{W},J)\setminus \left(\cM^z_{\lambda_2}(\widehat{W},J) \cup \cM^w_{\mu_2}(\widehat{W},J)\right) \rightarrow E_{\nu}$$
$$\tilde{u} \mapsto (e^{U}(m_w\cdot t) - e^{V}(m_z\cdot t + \frac{i}{m_z\cdot m_w}))$$
where $e^U$ and $e^V$ are the eigenvectors appearing in the expansion of the asymptotic representatives(as in  Theorem~\ref{asymp_rep1}) of $[\tilde{u};z]$ and $[\tilde{u};w]$ respectively. Then every pseudoholomorphic curve in 
$$\left( \cM^*(\widehat{W},J)\setminus \left(\cM^z_{\lambda_2}(\widehat{W},J) \cup \cM^w_{\mu_2}(\widehat{W},J)\right) \right) \setminus (\bigcup\limits_{i=1}^{m_z\cdot m_w}h_i^{-1}(0))$$
has $\delta_{\infty}([\tilde{u}; z],[\tilde{u};w])=0$. 
\end{itemize}
This detailed description allows us to present a theorem following the same approach as Theorem~\ref{delta1=0}.

\begin{theorem} \label{delta_rel=0}
For any pair $z,w\in \Gamma_+$, there exist a comeagre subset of $\cJ_O$ such that for every $J$ in that comeagre subset, the subset of pseudoholomorphic curves in $\cM^*(\widehat{W},J)$ with $\delta_{\infty}([\tilde{u}; z], [\tilde{u};w])=0$ contains an open dense subset.

\end{theorem}

\begin{proof}
The proof can be divided into two parts. Let's start with a simpler case where 
$$m_w\cdot \sigma^-_{\max}(\gamma_z^{m_z}) \neq m_z \cdot \sigma^-_{\max}(\gamma_w^{m_w}).$$
In this case, the submanifold $\cM^*(\widehat{W},J)\setminus \left(\cM^z_{\lambda_2}(\widehat{W},J) \cup \cM^w_{\mu_2}(\widehat{W},J)\right)$ is an open dense subset of the moduli space $\cM^*(\widehat{W},J)$ and every pseudoholomorphic curve in it has $\delta_{\infty}([\tilde{u}; z], [\tilde{u};w])=0$.\\

Now, let's consider the case where $m_w\cdot \sigma^-_{\max}(\gamma_z^{m_z}) = m_z \cdot \sigma^-_{\max}(\gamma_w^{m_w})$. In this case, the submanifold $\cM^*(\widehat{W},J)\setminus \left(\cM^z_{\lambda_2}(\widehat{W},J) \cup \cM^w_{\mu_2}(\widehat{W},J)\right)$ might still contain curves with $\delta_{\infty}([\tilde{u}; z], [\tilde{u};w])\neq0$. But, the pseudoholomorphic curves in the open subset 
$$\left( \cM^*(\widehat{W},J)\setminus \left(\cM^z_{\lambda_2}(\widehat{W},J) \cup \cM^w_{\mu_2}(\widehat{W},J)\right) \right) \setminus (\bigcup\limits_{i=1}^{m_z\cdot m_w}h_i^{-1}(0))$$
have $\delta_{\infty}([\tilde{u}; z], [\tilde{u};w])=0$. Thus it would be sufficient to prove that the maps $h_{i}$'s are regular as this would mean the submanifolds $h_{i}^{-1}(0)$ have positive codimension. This regularity is proved the same way as Theorem~\ref{regularity of evaluation map} with minor modifications to the proof.
\end{proof}

We now state the main theorems of this paper combining Theorems~\ref{delta1=0}, \ref{delta2=0} and \ref{delta_rel=0}. 

\begin{theorem}\label{maintheorem1}
There exists a comeagre subset $\cJ_{\mathrm{reg}} \subset \cJ_O$ such that for every $J \in \cJ_{\mathrm{reg}}$ the subspace of the moduli space $\cM^*(\widehat{W},J)$ consisting of the curves $\tilde{u}$ whose asymptotic self intersection index $\delta_{\infty}(\tilde{u})$ vanishes contains an an open dense subset.
\end{theorem}

From the proof of Theorem~\ref{delta_rel=0} and the definition of the intersection number of punctured pseudoholomorphic curves given in Section~\ref{intersection number and singularity index}, we can get the following theorem.

\begin{theorem}\label{maintheorem2}
There exists a comeagre subset $\cJ_{\mathrm{reg}} \subset \cJ_O$ such that for every $J\in \cJ_{\mathrm{reg}}$, the subset of the moduli space $\cM^*(\widehat{W},J) \times \cM^*(\widehat{W},J) $ consisting of the curves $(\tilde{u},\tilde{v})$ whose asymptotic contribution to the intersection $\tilde{u}*\tilde{v}$ vanishes contains an open dense subset.
\end{theorem}

The above theorem still holds for curves in the product of moduli spaces defined on potentially two different collections of non-degenerate Reeb orbits, homology classes etc.\\

It is crucial to emphasize that the proofs and characterizations provided in the Theorems~\ref{delta1=0}, \ref{delta2=0} and \ref{delta_rel=0} hold true for the moduli space $\cM^*(M,J)$ as well, resulting in the following theorems. As before, we assume that the complement of the open subset $O\subset M$ contain the Reeb orbits involved in the definition of the moduli space.

\begin{theorem}\label{maintheorem3}
There exists a comeagre subset $\cJ_{\mathrm{reg}} \subset \cJ_O$ such that for every $J \in \cJ_{\mathrm{reg}}$ the subspace of the moduli space $\cM^*(M,J)$ consisting of the curves $\tilde{u}$ whose asymptotic self intersection index $\delta_{\infty}(\tilde{u})$ vanishes contains an an open dense subset.
\end{theorem}

\begin{theorem}\label{maintheorem4}
There exists a comeagre subset $\cJ_{\mathrm{reg}} \subset \cJ_O$ such that for every $J \in \cJ_{\mathrm{reg}}$, the subspace of the moduli space $\cM^*(M,J) \times \cM^*(M,J) $ containing the curves $(\tilde{u},\tilde{v})$ whose asymptotic contribution to the intersection $\tilde{u}*\tilde{v}$ vanishes contains an an open dense subset.
\end{theorem}

As before, the above theorem still holds for curves in the product of moduli spaces defined on potentially two different collections of non-degenerate Reeb orbits, homology classes etc.

\printbibliography
\end{document}